\documentclass[11pt,a4paper,twoside]{article}

\usepackage{times}
\usepackage{amsmath,amssymb,amsthm,amsfonts}
\usepackage{a4wide}
\usepackage{exscale}   
\usepackage{enumerate}
\usepackage{latexsym, mathrsfs}
\usepackage[latin1]{inputenc} 
\usepackage[T1]{fontenc}

\usepackage[all]{xy} 
\usepackage{srcltx} 

\usepackage{stmaryrd}

\usepackage{fixme}

\newcommand{\restr}{\rvert}   
\newcommand{\norm}[1]{\left\lVert#1\right\rVert}   

\newcommand{\demph}[1]{{\it #1}}

\DeclareMathOperator{\kernel}{Kern}

\DeclareMathOperator{\Hom}{Hom}

\DeclareMathOperator{\Mat}{M} 

\DeclareMathOperator{\id}{Id}
\DeclareMathOperator{\ev}{ev} 

\newcommand{\op}{\ensuremath{\text{\rm op}}}

\newcommand{\Cont}{{\mathscr C}} 

\DeclareMathOperator{\Lin}{L} 
\DeclareMathOperator{\Ad}{Ad}

\DeclareMathOperator{\Komp}{K} 

\DeclareMathOperator{\KTh}{K}
\DeclareMathOperator{\KK}{KK}

\newcommand{\C}{\ensuremath{{\mathbb C}}}
\newcommand{\E}{\ensuremath{{\mathbb E}}}

\newcommand{\M}{\ensuremath{{\mathbb M}}}
\newcommand{\N}{\ensuremath{{\mathbb N}}}

\newcommand{\Z}{\ensuremath{{\mathbb Z}}}

\newcommand{\mA}{\ensuremath{{\mathcal A}}}

\newcommand{\mC}{\ensuremath{{\mathcal C}}}

\newcommand{\mE}{\ensuremath{{\mathcal E}}}
\newcommand{\mF}{\ensuremath{{\mathcal F}}}
\newcommand{\mG}{\ensuremath{{\mathcal G}}}

\newcommand{\mK}{\ensuremath{{\mathcal K}}}

\newcommand{\mM}{\ensuremath{{\mathcal M}}}

\newcommand{\mT}{\ensuremath{{\mathcal T}}}

\newcommand{\fA}{\ensuremath{{\mathfrak A}}}
\newcommand{\fB}{\ensuremath{{\mathfrak B}}}

\theoremstyle{plain}

\newtheorem {theorem} {Theorem}[section]
\newtheorem {lemma}[theorem] {Lemma}
\newtheorem {proposition} [theorem]{Proposition}
\newtheorem {corollary} [theorem]{Corollary}

\newtheorem* {theorem*} {Theorem}
\newtheorem* {proposition*} {Proposition}
\newtheorem* {lemma*}{Lemma}

\theoremstyle{definition}

\newtheorem {definition} [theorem]{Definition}

\newtheorem {deflemma} [theorem]{Definition and Lemma}

\newtheorem {remark} [theorem]{Remark}

\newtheorem {question} [theorem]{Question}

\newtheorem {Xample}[theorem] {Example}

\DeclareMathOperator{\KKban}{KK^{\ban}}

\DeclareMathOperator{\Eban}{\E^{\ban}}

\DeclareMathOperator{\ban}{ban}

\DeclareMathOperator{\kk}{kk}

\DeclareMathOperator{\kkban}{kk^{\ban}}

\DeclareMathOperator{\kkbanfunc}{kk^{\ban}}

\newcommand{\LazyAnd}{\quad \text{and} \quad}

\newcommand{\unital}[1]{\widetilde{#1}} 

\newcommand{\zylinder}{{\mathrm Z}} 

\newcommand{\cone}{{\mathrm C}}  

\newcommand{\suspension}{{\mathrm \Sigma}}  
\newcommand{\hotimes}{{\hat{\otimes}}} 

\newcommand{\cT}{\mT}

\DeclareMathOperator{\Mban}{\M^{\ban}} 

\DeclareMathOperator{\Moritaban}{Mor^{\ban}} 
\DeclareMathOperator{\Moritabancat}{\cat{Mor}^{\ban}} 

\DeclareMathOperator{\colim}{colim}

\newcommand{\cat}[1]{\text{\rm \bf  #1}}

\newcommand{\Ab}{\ensuremath{\cat{Ab}}}

\newcommand{\BanSp}{\ensuremath{\cat{BanSp}}}

\newcommand{\BanAlg}{\ensuremath{\cat{BanAlg}}}

\newcommand{\Cpt}{\ensuremath{\cat{Cpt}}}
\newcommand{\CW}{\ensuremath{\cat{CW}}}

\DeclareMathOperator{\pt}{pt.}

\DeclareMathOperator{\SW}{\Sigma Ho}
\DeclareMathOperator{\SWban}{\Sigma Ho^{\ban}}
\DeclareMathOperator{\SigmaHo}{\Sigma Ho}

\newcommand{\SWbancat} {\ensuremath{\mathbf{\Sigma}\cat{Ho}^{\ban}}}

\DeclareMathOperator{\EHo}{EHo}
\DeclareMathOperator{\EHoban}{EHo^{\ban}}
\DeclareMathOperator{\EHobanfunc}{EHo^{\ban}}

\newcommand{\EHobancat} {\ensuremath{\cat{EHo}^{\ban}}}

\DeclareMathOperator{\can}{can}

\begin{document}

\title{$kk$-Theory for Banach Algebras I:\\ The Non-Equivariant Case}
\author{Walther Paravicini}
\date{\today}
\maketitle

\begin{abstract}
\noindent kk$^{\text{ban}}$ is a bivariant K-theory for Banach algebras that has reasonable homological properties, a product and is Morita invariant in a very general sense. We define it here by a universal property and ensure its existence in a rather abstract manner using triangulated categories. The definition ensures that there is a natural transformation from Lafforgue's theory KK$^{\text{ban}}$ into it so that one can take products of elements in KK$^{\text{ban}}$ that lie in kk$^{\text{ban}}$.

\medskip 


\medskip

\noindent AMS 2010 {\it Mathematics subject classification:} Primary 19K35 ; Secondary 46M18, 46M20, 46H99
\end{abstract}

\noindent Cuntz's bivariant $\KTh$-theory $\kk$ is defined on the category of bornological (or locally convex) algebras \cite{Cuntz:97, CMR:07}. Restricting $\kk$ to the category of Banach algebras leads to a theory that has nice homological properties and a product but lacks Morita invariance. Yet Morita invariance (in Lafforgue's Banach algebraic sense presented in \cite{Paravicini:07:Morita:richtigerschienen}) is important if we want to compare $\kk$ to the bivariant $\KTh$-theory $\KKban$ for Banach algebras introduced by Lafforgue \cite{Lafforgue:02}; note that this theory is Morita invariant in the second component \cite{Paravicini:07:Morita:richtigerschienen}. In order to admit a natural transformation from $\KKban$ to $\kk$, it is thus necessary that the right-hand side be Morita invariant. We hence define a suitable version of $\kk$ for Banach algebras, called $\kkban$. 

A first attempt to define such a theory would be to adapt the definitions of \cite{Cuntz:97} or \cite{CMR:07} directly. This means to substitute the algebras of ``compact operators'' that are used as stabilisations in the several definitions of $\kk$ with a limit over all possible  Morita equivalences of Banach algebras. But the relation of the non-commutative suspension functor $J$ used by Cuntz and this stabilisation is somewhat unclear and this ansatz seems to lead to a dead end.

In principle, it is still possible -- though somewhat forced -- to adapt Cuntz's general theory by using its triangulated structure. 
But it seems cleaner and more conceptual to use the so-called Spanier-Whitehead construction -- without furthur ado -- as laid out in a slightly different framework in the appendix of \cite{DellAmbrogio:08}. It does not make use of Cuntz's non-commutative suspension functor $J$ but just of the ordinary suspension functor which we call $\Sigma$. Several paragraphs of the present article are transferred rather directly from \cite{DellAmbrogio:08} into our context. The functor $J$ is analysed a posteriori in Paragraph \ref{Subsection:FunctorJ}.

In this article, we define several theories with more and more desirable properties: A theory $\SWban$ which comes right out of the Spanier-Whitehead construction and can be thought of as a stable homotopy category of Banach algebras. As a ``quotient'' of this theory we obtain a theory $\EHoban$ which is comparable to the theory $\SigmaHo$ of \cite{CMR:07} and has long exact sequences in both variables for semi-split extensions of Banach algebras. Finally, we define $\kkban$ as a ``quotient'' of $\EHoban$; it is a Morita invariant theory. The notion of Morita equivalence that we use here is an extension of Lafforgue's notion to possibly degenerate Banach algebras (i.e., Banach algebras $B$ such that the span of $B\cdot B$ is not dense in $B$).

We establish some properties of $\kkban$, namely its action on $\KTh$-theory, Bott periodicity and the fact that $\kkban(\C,B) \cong \KTh_0(B)$ for every Banach algebra $B$.

We also construct a natural transformation from $\KKban$ to $\kkban$ in the even as well as in the odd case. Moreover, we show that the canonical natural transformation from $\KK$ to $\kkban$ respects the Kasparov product.

The next step will be to generalise these constructions to the equivariant case. The definition of an equivariant version of $\kkban$ itself is straightforward but in order to construct the natural transformation on $\KK^{\ban}_G$ one has to take several technical subtleties into account. So we postpone the equivariant case to an upcoming article to keep the exposition clearer. With an equivariant theory it will be possible to present the Bost conjecture with Banach algebra coefficients in a well-adapted environment.


Another extension of the present results would be to consider other categories than Banach algebras. The work of Dell'Ambrogio \cite{DellAmbrogio:08} is already formulated quite generally, though not suited without changes for Banach algebras; I've tried to give the basic definitions in this article in a way which can easily be rephrased in a very general setting of some type of ``exact categories of algebras''. Such a general framework would not only be aesthetically pleasing but also helpful when considering categories such as $\Cont_0(X)$-Banach algebras etc.

I would like to thank Joachim Cuntz for his ample supply of ideas, Andreas Thom for hinting me to the thesis of Ivo Dell'Ambrogio, and Ivo for a pleasant and instructive morning in Z\"{u}rich. I would also like to thank Siegfried Echterhoff for his support and Martin Grensing for his careful reading and his helpful suggestions. This work was supported by the SFB 878 of the Deutsche Forschungsgemeinschaft in M\"unster.

\section*{Notation}

Let $\BanSp$ denote the category of Banach spaces and continuous linear maps and let $\Ab$ denote the category of abelian groups and group homomorphisms. Let $\BanAlg$ denote the category of Banach algebras and continuous homomorphisms. A \demph{short exact sequence in $\BanAlg$} or \demph{extension} is a sequence
\[
\xymatrix{
A\ar@{>->}[r]^{i} & B \ar@{->>}[r]^{p}  & C
}
\]
where $i$ is injective (with closed image), $p$ is surjective and $i(A) = \kernel (p)$. 

Let $\mE_{\max}$ denote the class of all such short exact sequences. Let $\mE_{\min}$ denote the class of all short exact sequences which have a bounded linear split; these extensions are called \demph{semi-split extensions}.

\medskip 

For every locally compact space $X$ and every Banach space (or Banach algebra) $A$ define $AX$ as the Banach space (Banach algebra) $\Cont_0(X,  A)$; if $x\in X$, then $\ev_x^A\colon A X \to A$ denotes the evaluation homomorphism at $x$.

Define, for every Banach algebra $A$,
\begin{eqnarray*}
\zylinder{A} &:= & A[0,1]\\
\cone A & := & A[0,1[\\
\Sigma A & := & A]0,1[ 
\end{eqnarray*}

In the case $A = \C$ we just write $\zylinder$,  $\cone$ and $\suspension$ for the Banach algebras $\zylinder \C$, $\cone \C$ and $\suspension \C$, respectively.

The cone extension is defined as
\[
\xymatrix{
\Sigma A\ar@{>->}[r] & \cone A \ar@{->>}[r]^{\ev^A_0}  & A.
}
\] 
It is semi-split.

If $\varphi \colon A \to B$ is a continuous homomorphism of Banach algebras then the \demph{mapping cone} $\cone_\varphi$ of $\varphi$ is defined as the pullback in the diagram
$$
\xymatrix{
\cone_{\varphi} \ar[r]^-{\epsilon(\varphi)} \ar[d] & A \ar[d]^{\varphi} \\
\cone B \ar[r]^-{\ev^B_0} & B.
}
$$
This square fits into a diagram
$$
\xymatrix{
\Sigma B \ar@{>->}[r]^-{\iota(\varphi)} \ar[d]_= &\cone_{\varphi} \ar@{->>}[r]^-{\epsilon(\varphi)} \ar[d] & A \ar[d]^{\varphi} \\
\Sigma B \ar@{>->}[r] &\cone B \ar@{->>}[r]^-{\ev^B_0} & B.
}
$$
As the lower line is semi-split, also the upper line is semi-split; it is called the \demph{cone extension of $\varphi$}.

Two parallel morphisms $\varphi_0, \varphi_1 \colon A \to B$ are said to be \demph{homotopic} if there exists a homotopy from $\varphi_0$ to $\varphi_1$, i.e., a morphism $\varphi\colon A \to B[0,1]$ such that $\ev_0^B \circ \varphi = \varphi_0$ and $\ev_1^B \circ \varphi = \varphi_1$.

Note that homotopy is an equivalence relation on $\Hom(A,B)$, compatible with composition.


Let $B$ be a Banach algebra. As in \cite{Lafforgue:02}, a \demph{(Banach) $B$-pair} $E$ is a pair $E=(E^<,E^>)$, where $E^<$ is a left Banach $B$-module and $E^>$ is a right Banach $B$-module, endowed with a bilinear bracket $\langle \cdot,\cdot\rangle\colon E^< \times E^> \to B$ satisfying the following conditions:
\begin{itemize}
\item $\forall b\in B\ \forall e^< \in E^<\ \forall e^> \in E^>:\ \langle b e^<, e^>\rangle = b \langle e^<, e^>\rangle$.
\item $\forall b\in B\ \forall e^< \in E^<\ \forall e^> \in E^>:\ \langle e^<, e^>b\rangle = \langle e^<, e^>\rangle b$.
\item $\forall e^< \in E^<\ \forall e^> \in E^>:\ \norm{\langle e^<,e^> \rangle} \leq \norm{e^<}\norm{e^>}$.
\end{itemize}

A Banach pair is a Banach algebraic replacement for a Hilbert module over a C$^*$-algebra. There are notions of ``adjoinable operators'' $\Lin(E)$ and ``compact operators'' $\Komp(E)$ if $E$ is a Banach $B$-pair, see \cite{Lafforgue:02}. If $A$ and $B$ are Banach algebras, then a \demph{Banach $A$-$B$-pair} is a Banach $B$-pair together with a morphism of Banach algebras from $A$ to $\Lin(E)$ (so $A$ ``acts on $E$ from the left''). 

\medskip 
If $B$ is a Banach algebra, then we mean by $BB$ or $B^2$ the closed linear span in $B$ of all twofold products of elements of $B$. Expressions like $BE^<$ and $\langle E^<,E^>\rangle$ should be interpreted similarly where $E$ is a Banach $B$-pair.

A Banach algebra $B$ is called \demph{non-degenerate} if $BB = B$. A Banach $B$-pair $E$ is called \demph{full} if $\langle E^<,E^>\rangle = B$; it is called \demph{non-degenerate} if $BE^< = E^<$ and $E^>B = E^>$.

If $E$ and $F$ are Banach spaces, then $E\otimes F$ denotes the \demph{completed projective tensor product} $E \hotimes_{\pi} F$ of $E$ and $F$. If $A$ and $B$ are Banach algebras, then $A \otimes B$ is again a Banach algebra in a canonical way. 

\section{Morita equivalences for Banach algebras}

\subsection{Morita equivalences for non-degenerate Banach algebras}

\begin{definition}[Morita equivalence]\label{Definition:MoritaEquivalence:nondegenerate} Let $A$, $B$ be non-degenerate Banach algebras. A \demph{Morita equivalence} between $A$ and $B$ is a pair $\left({_BE^<_A},{_AE^>_B}\right)$ endowed with a bilinear bracket $\langle\cdot,\cdot \rangle_B\colon E^< \times E^> \to B$ and a bilinear bracket ${_A}\langle\cdot,\cdot\rangle\colon E^>\times E^< \to A$ satisfying the following conditions:
\begin{enumerate}
\item $(E^<,E^>)$ with $\langle\cdot,\cdot \rangle_B$ is an $A$-$B$-pair.
\item $(E^>,E^<)$ with ${_A}\langle\cdot,\cdot\rangle$ is a $B$-$A$-pair.
\item The two brackets are compatible:
\[
\langle e^<,e^>\rangle_B f^< = e^< {_A}\langle e^>,f^<\rangle \LazyAnd e^> \langle f^<,f^>\rangle_B = {_A}\langle
e^>,f^<\rangle f^>
\]
for all $e^<, f^< \in E^<$ and $e^>,f^> \in E^>$.
\item The pairs $(E^<,E^>)$ and $(E^>,E^<)$ are full and non-degenerate.
\end{enumerate}
$A$ and $B$ are called \demph{Morita equivalent} if there is a Morita equivalence between $A$ and $B$.
\end{definition}

A Morita equivalence $E$ between $A$ and $B$ gives rise to a \demph{linking algebra} $L=A\oplus E^> \oplus E^< \oplus B$ which is usually represented as
\[
L=\left(\begin{array}{cc} A& E^>\\ E^< & B\end{array}\right)
\]
with a multiplication that operates as the multiplication of two-by-two matrices; note that the above definition implies that all the binary operations $A \times E^>\to E^>$,\ \ $E^< \times E^> \to B$ etc.\ are continuous and non-degenerate in the sense that $AE^> = E^>$ etc. In particular, it is easy to see that if $A$ is Morita equivalent to $B$ in the above sense, then both, $A$ and $B$, have to be non-degenerate Banach algebras, automatically, even if we initially exclude this condition. So this notion of Morita equivalence is limited to non-degenerate Banach algebras and has to be reformulated in order to give something useful for degenerate algebras.

On the other hand, Morita equivalence in the sense given above turns out to be an equivalence relation on the sub-category of non-degenerate Banach algebras which has the following remarkable property:

\begin{theorem}
Let $A$ and $B$ be non-degenerate Banach algebras. If $A$ is Morita equivalent to $B$, then $\KTh_*(A) \cong \KTh_*(B)$. 
\end{theorem}

See \cite{Paravicini:07:Morita:richtigerschienen} for a proof of this fact first observed by Vincent Lafforgue; it relies on \cite{Lafforgue:02} and the bivariant $\KTh$-theory $\KKban$ introduced therein. In fact, the main result of \cite{Paravicini:07:Morita:richtigerschienen} is more general:

\begin{theorem}
Let $A$, $B$ and $C$ be non-degenerate Banach algebras. If $A$ is Morita equivalent to $B$, then $\KKban(C,A) \cong \KKban(C,B)$. 
\end{theorem}

\subsection{Morita equivalences for possibly degenerate Banach algebras}

Because we want to be able to define $\kkban$ also for Banach algebras which might not be non-degenerate, we have to extend the concept of Morita equivalences to those Banach algebras. An example for a degenerate Banach algebra is the non-unital tensor algebra $T_1X$ for a Banach space $X$ as introduced in Definition \ref{Definition:BanachTensorAlgebra}.

Let $A$ and $B$ be Banach algebras. The following definition is inspired by a suggestion of Joachim Cuntz.

\begin{definition}\label{Definition:MoritaEquivalence:degenerate} Let $A$, $B$ be Banach algebras. A \demph{Morita equivalence} between $A$ and $B$ is a pair $\left({_BE^<_A},{_AE^>_B}\right)$ endowed with a bilinear bracket $\langle\cdot,\cdot \rangle_B\colon E^< \times E^> \to B$ and a bilinear bracket ${_A}\langle\cdot,\cdot\rangle\colon E^>\times E^< \to A$ satisfying the conditions 1., 2. and 3. of Definition \ref{Definition:MoritaEquivalence:nondegenerate} as well as 
\begin{itemize}
\item[4.'] there are $k,l\in \N$ satisfying $A^k \subseteq {_A}\langle E^>,E^<\rangle$ and $B^l\subseteq \langle E^<, E^>\rangle_B$.
\end{itemize}
%
$A$ and $B$ are called \demph{Morita equivalent} if there is a Morita equivalence between $A$ and $B$.
\end{definition}

Note that this definition is equivalent to Definition \ref{Definition:MoritaEquivalence:nondegenerate} in case that $A$ and $B$ are non-degenerate (with the slight change that if $E$ is a Morita equivalence in the sense of \ref{Definition:MoritaEquivalence:degenerate} then $(BE^<A, AE^>B)$ is a Morita equivalence in the sense of  \ref{Definition:MoritaEquivalence:nondegenerate}).

\begin{proposition}
Morita equivalence is an equivalence relation on the class of all Banach algebras.
\end{proposition}
\begin{proof}
Every Banach algebra $A$ is clearly Morita equivalent to itself, use $E^<=E^>=A$ and $k=l=2$. Symmetry is also obvious. Transitivity can be shown using the balanced tensor product of Banach modules.
\end{proof}

\begin{Xample}
\begin{enumerate}\item Let $B$ be a Banach algebra, degenerate or not, and $n\in \N$. Then the matrix algebra $\Mat_n(B)$ of $n\times n$-matrices with entries in $B$ is Morita equivalent to $B$. A Morita equivalence is given by $(B^{1\times n},B^{n \times 1})$ with the obvious inner products. Note that, for example, the $B$-valued inner product takes it values in $BB$.
\item If $B$ is a Banach algebra and $n\in \N$, then $B^n$ and $B$ are Morita equivalent. A Morita equivalence is given by $(B, B^{n-1})$. Of course, there are other ways to analyse this situation: $B^n$ is an ideal in $B$, and the quotient is nilpotent.
\item Let $E$ be a Banach $B$-pair where $B$ is some Banach algebra. Then $E$ is a Morita equivalence between $\Komp_B(E)$ and $B':=\langle E^<, E^>\rangle_B$. Note that even if $B$ is non-degenerate and so is $E$, $\Komp_B(E)$ does not have to be. This happens if, in addition, $E$ is full, i.e., if $B=B'$. However, if $E$ is not full, we have still a Morita equivalence between $\Komp_B(E)$ and the ideal $B'$ of $B$.
\end{enumerate}
\end{Xample}

If $A$ and $B$ are Morita equivalent, then we want to make sure that they are also Morita equivalent to the linking algebra connecting them. This is not completely obvious and the proof given here highlights the kind of methods one might want to use when manipulating Morita equivalences in the sense introduced above.

\begin{proposition}\label{Proposition:MoritaEquivarianceOfLinkingAlgebra}
Let $E=(E^<, E^>)$ be a Morita equivalence between $A$ and $B$. Form its linking algebra $L$ as above. Define $F:=E\oplus B$, that is, $F^>:= E^>\oplus B \subseteq L$  and $F^< := E^< \oplus B \subseteq L$. Note that $F$ is, in a straightforward manner, a Banach $B$-pair that comes with a left-action of $L$ and a compatible $L$-valued inner product coming from multiplication in $L$. There is an $m$ in $\N$ such that $L^m \subseteq F^>F^<$ and $B^2\subseteq F^<F^>$. So $F$ is a Morita equivalence between $L$ and $B$. There is a similar Morita equivalence between $L$ and $A$.
\end{proposition}
\begin{proof}
Choose $k,l\in \N$ such that $A^k\subseteq E^>E^<$ and $B^l\subseteq E^<E^>$. We first analyse the term $L^m \cap B$ for given $m$. We have to show that it is contained in $F^<F^>\cap B = B^2$ for large enough $m$. So we consider a non-vanishing product of $m$ elements of $L$ that lies in $B$. Without loss of generality we can assume that all factor are contained in $A$, $B$, $E^<$ or $E^>$. Because the product does not vanish, subsequent factors have to match in the sense that a factor in $B$ is not followed by a factor in $E^>$ or $A$ etc.. 

Now we can distinguish several cases depending on where the factors are from, keeping in mind that the final result has to be an element of $B$. 

If two or more of the factors are in $B$ we can take the product of the remaining factors between these factors, and all together this will give at least to factors in $B$ so we are done. If two or more of the factors are in $E^<$ then there are the same number of factors in $E^>$ and we can multiply everything between one factor in $E^<$ and the next one in $E^>$, including the elements of $E^<$ and $E^>$ at the endpoints, to get at least two factors in $B$. If $2k$ more of the factors are in $A$, then we can multiply $k$ of them (and what is between them) to obtain at least two elements in $E^>E^<$. Then we are in the preceding case. So $m=1+1+1+(2k-1)+1$ will do.

A similar reasoning shows that there is an $m\in \N$ such that $L^m\cap E^< \subseteq BE^< = F^>F^<\cap E^<$ and $L^m \cap E^> \subseteq E^>B= F^>F^<\cap E^>$ and $L^m \cap A\subseteq E^> E^< = F^>F^< \cap A$. This shows that there is an $m\in \N$ such that $L^m \subseteq F^>F^<$.
\end{proof}

\subsection{Morita equivalences and Morita morphisms}

Morita equivalences of Banach algebras fit nicely in the category of ``Morita morphisms'', an easy-to-define category that can be thought of as homotopy classes of compositions of ordinary homomorphisms and Morita equivalences. This category was introduced and used in \cite{Paravicini:07:Morita:richtigerschienen} for the case of non-degenerate Banach algebras, and in order to avoid a lengthy discussion we only give here the definitions, the necessary results and some remarks on how to prove them in the setting of possibly degenerate Banach algebras.

So let $A$ and $B$ be Banach algebras.

\begin{definition}[Morita cycle, {\it{cf.}} Definition 5.7 of \cite{Paravicini:07:Morita:richtigerschienen}:]
A \demph{Morita cycle} $(F,\varphi)$ from $A$ to $B$ is a Banach $B$-pair $F$ together with a homomorphism $\varphi\colon A \to \Lin_B(F)$ such that there is a $k\in \N$ satisfying $\varphi(A)^k \subseteq \Komp_B(F)$; we will sometimes simply write $F$ for this cycle and suppress the left action $\varphi$ in the notation. The class of all Morita cycles from $A$ to $B$ will be denoted by $\Mban(A,B)$.
\end{definition}

One can compose Morita cycles using the tensor product: If $F\in \Mban(A,B)$ is a Morita cycle from $A$ to $B$ and $F'\in \Mban(B,C)$ is a Morita cycle from $B$ to $C$, then $F\otimes_B F'$ is a Morita cycle from $A$ to $C$. If $\varphi \colon A \to B$ is a homomorphism of Banach algebras, then the Banach $B$-pair $(B,B)$ equipped with the left action of $A$ induced from $\varphi$ is a Morita cycle from $A$ to $B$, denoted by $\Mban(\varphi)$. And if $E$ is a Morita equivalence from $A$ to $B$, then it can also be regarded as a Morita cycle $\Mban(E)$ from $A$ to $B$.

However, if we consider the the composition of Morita cycles up to isomorphism, it does not give us a category; for example, $B \otimes_B B$ is, in general, not isomorphic to $B$ as a $B$-$B$-bimodule. But up to homotopy, everything works fine and Morita equivalences give isomorphisms in the resulting category:

\begin{definition}[Homotopy of Morita cycles]\label{Definition:HomotopyOfMoritaCycles}
Let $A$ and $B$ be Banach algebras and let $(F_0,\varphi_0), (F_1, \varphi_1)$ be Morita cycles from $A$ to $B$. A homotopy from $F_0$ to $F_1$ is a Morita cycle $(F, \varphi)$ from $A$ to $B[0,1]$ such that $F$ is also a $\Cont[0,1]$-Banach $B[0,1]$-pair such that $\varphi(a)$ is $\Cont[0,1]$-linear for all $a\in A$ and such that the fibre of $(F,\varphi)$ at $0$ and $1$ are isomorphic to $(F_0,\varphi_0)$ and $(F_1, \varphi_1)$, respectively. We write $(F_0,\varphi_0) \sim(F_1, \varphi_1)$ in this case.
\end{definition}

That $F$ is a $\Cont[0,1]$-Banach $B[0,1]$-pair means that the actions of $B[0,1]$ on $F^<$ and $F^>$ extend to unital actions of $\unital{B}[0,1]$ and that $F$ has a $B[0,1]$-valued and $\unital{B}[0,1]$-bilinear inner product, see \cite{Paravicini:10:GreenJulg:submitted} and compare the Remarque following Th\'{e}or\`{e}me 1.2.8 of \cite{Lafforgue:02}.

Note that in the above definition, the fibre of $(F,\varphi)$ over some $t\in [0,1]$ is constructed as follows: The spaces $F^<$ and $F^>$ are unital Banach modules over the algebra $\Cont[0,1]$ so it makes sense to speak of their fibres $F_t^<$ and $F_t^>$, respectively. These spaces form a $B$-pair $F_t$ in a canonical way, where we identify $B$ with the fibre over $t$ of the $\Cont[0,1]$-Banach algebra $B[0,1]$. If $a\in A$, then $\varphi(a)$ is a $\Cont[0,1]$-linear operator on $F$, so it factors to an operator on $F_t$, call it $\varphi_t(a)$; it is clear that $(F_t, \varphi_t)$ is a Morita cycle from $A$ to $B$.

The proofs of the subsequent statements of this section can be carried out as in the non-degenerate case, see Section 5.4 of \cite{Paravicini:07:Morita:richtigerschienen}, with some minor changes; for example, in the proof of the statement corresponding to Lemma 5.18 of \cite{Paravicini:07:Morita:richtigerschienen} you have to consider $S\in \Komp_B(E)^3$ instead of $S\in \Komp_B(E)$. 

\begin{lemma}\label{Lemma:MoritabanHomotopies}
Let $A$, $B$ and $C$ be Banach algebras.
\begin{enumerate}
\item if $\varphi \colon A \to B$ is a continuous homomorphism of Banach algebras and if $(F,\psi)\in \Mban(B,C)$, then we have a homotopy
$$
\Mban(\varphi) \otimes_{\psi} F = (B \otimes_{\psi} F, \varphi \otimes 1) \ \sim \ (F, \psi \circ \varphi)=: \varphi^*(F,\psi);
$$
\item if $(F,\varphi) \in \Mban(A,B)$ and $\psi \colon B \to C$ is a continuous homomorphism of Banach algebras, then we have a homotopy
$$
F \otimes_{B} \Mban(\psi) = F \otimes_\psi C \ \sim \ F \otimes_{\unital{\psi}} \unital{C} =: \psi_*(F, \varphi).
$$
\end{enumerate}
\end{lemma}

Because homotopy is an equivalence relation on the class of all Morita cycles from $A$ to $B$, we have everything in place for the following definition and theorem:

\begin{definition}[Morita morphism]
A \demph{Morita morphism} is a homotopy class of Morita cycles. If $\varphi \colon A \to B$ is a continuous homomorphism of Banach algebras, then $\Moritaban(\varphi)\in \Moritaban(A,B)$ denotes the homotopy class of $\Mban(\varphi) \in \Mban(A,B)$; similarly, if $E$ is a Morita equivalence between $A$ and $B$, then $\Moritaban(E)\in \Moritaban(A,B)$ denotes the homotopy class of the corresponding cycle $\Mban(E)\in \Mban(A,B)$.
\end{definition}

\begin{theorem}
Morita morphisms form a category $\Moritabancat$ (under the flipped tensor product of representatives) such that
\begin{enumerate}
\item the map $\varphi \mapsto \Moritaban(\varphi)$ is a functor from $\BanAlg$ to $\Moritabancat$;
\item  if $E$ is a Morita equivalence from $A$ to $B$ then $\Moritaban(E)$ is an isomorphism from $A$ to $B$.
\end{enumerate}
\end{theorem}

\noindent The following definition and lemma will be useful in the subsequent sections:

\begin{definition}\label{Definition:ConcurrentHomomorphismOfMoritaEQuivalences}
Let $\chi \colon A \to A'$ and $\psi \colon B \to B'$ be homomorphisms of Banach algebras. Let $E$ be a Morita equivalence between $A$ and $B$ and let $E'$ be a Morita equivalence between $A'$ and $B'$. Then a \demph{concurrent homomorphism} of Morita equivalences from $E$ to $E'$ (with coefficient maps $\chi$ and $\psi$) is a pair $\Phi = (\Phi^<, \Phi^>)$ such that $\Phi^> \colon E^> \to (E')^>$ and $\Phi^< \colon E^< \to (E')^<$ are continuous linear maps compatible with all products that come with the Morita equivalence, i.e.,
\smallskip

\begin{center}
\begin{tabular}{rclcrcl}
$\Phi^>(ae^>)$ &=& $\chi(a) \Phi^>(e^>)$, && $\Phi^<(e^< a)$ &=& $\Phi^<(e^<) \chi(a)$, \\ 
$\Phi^>(e^>b)$ &= & $\Phi^>(e^>) \psi(b)$, && $\Phi^<(be^<)$ &=& $\psi(b) \Phi^<(e^<)$,\\
  $\psi(\langle e^<, e^>\rangle_B)$ & =&$\langle \Phi^<(e^<), \Phi^>(e^>)\rangle_{B'}$&and& $\chi({_A}\langle e^>, e^<\rangle)$ &=&  ${_{A'}}\langle \Phi^>(e^>), \Phi^<(e^<) \rangle$
\end{tabular}
\end{center}
\smallskip

\noindent for all $a\in A$, $b\in B$, $e^< \in E^<$ and $e^> \in E^>$. 

We also write ${_{\chi}} \Phi_\psi$ for $\Phi$ to remind us of the ``coefficient maps'' that come with it.
\end{definition}

\begin{lemma}\label{Lemma:ConcurrentEquivalences}
In the situation of the preceding definition, the mapping cone of $\Phi$ induces a homotopy that yields the equation
$$
\Moritaban(\psi) \circ \Moritaban(E) = \Moritaban(E') \circ \Moritaban(\chi).
$$
\end{lemma}

\subsection{Morita invariant functors}

\begin{definition}
Let $\mC$ be a category and $\mF \colon \BanAlg \to \mC$ a functor.
\begin{itemize}
\item The functor $\mF$ is called \demph{homotopy invariant} if $\mF(\varphi) = \mF(\psi)$ whenever $\varphi$ and $\psi$ are homotopic continuous homomorphisms of Banach algebras.
\item The functor $\mF$ is called \demph{Morita invariant} if, whenever $E$ is a Morita equivalence between Banach algebras $A$ and $B$ and if 
$$
L = \left(\begin{matrix}A & E^>\\ E^< & B\end{matrix}\right).
$$
denotes the corresponding linking algebra, then the inclusion map $\iota$ of $B$ into $L$ induces an isomorphism $\mF(\iota)\colon \mF(B) \cong \mF(L)$.
\end{itemize}
\end{definition}

\begin{Xample}
The canonical functor $\Moritaban \colon \BanAlg \to \Moritabancat$ is homotopy invariant as well as Morita invariant.
\end{Xample}
\begin{proof}
Homotopy invariance is more or less built in the definition of $\Moritabancat$; we thus consider only Morita invariance: What we want to show is that the inclusion $\iota$ of $B$ in the linking algebra $L$ that comes with the equivalence $E$ is an isomorphism under $\Moritaban$. Recall from Proposition~\ref{Proposition:MoritaEquivarianceOfLinkingAlgebra} the pair $E \oplus B$ constitutes a Morita equivalence between $L$ and $B$. Let $\Phi\colon B \to E \oplus B$ be the canonical inclusion of Morita equivalence ${_B}B_B$ into the Morita equivalence ${_L}(E\oplus B)_B$; it is a concurrent homomorphism of Morita equivalences from $B$ to $E \oplus B$ with coeffient maps $\iota$ and $\id_B$. It follows from Lemma~\ref{Lemma:ConcurrentEquivalences} that
$$
\Moritaban(E\oplus B) \circ \Moritaban(\iota)  = \Moritaban(\id_B) \circ \Moritaban(B) = 1_B.
$$
We hence have $\Moritaban(\iota) = \Moritaban(E\oplus B)^{-1}$, so $\Moritaban(\iota)$ is an isomorphism.
\end{proof}

\noindent In the rest of this section, we are going to show the following theorem; if you translate it into the C$^*$-world, many steps can be greatly simplified (e.g., if $E$ is a Morita equivalence from $A$ to $B$, then $\Komp_B(E)$ is isomorphic to $A$ and $A^k=A$ for all $k\in \N$), but in the Banach world, it seems adequate to carry out the technical steps with some care.

\begin{theorem}\label{Theorem:LiftToMoritaban}
The functor $\Moritaban \colon \BanAlg \to \Moritabancat$ is the universal homotopy invariant and Morita invariant functor on $\BanAlg$: 
Let $\mC$ be a category and $\mF \colon \BanAlg \to \mC$ a functor. Then $\mF$ is homotopy invariant and Morita invariant if and only if $\mF$ factors through  $\Moritaban \colon \BanAlg \to \Moritabancat$. The factorisation, if it exists, is unique. 
\end{theorem}

If a functor factors through $\Moritaban \colon \BanAlg \to \Moritabancat$ then it clearly has to be homotopy invariant and Morita invariant because $\Moritaban$ is. So we have to show the opposite direction: Let $\mF \colon \BanAlg \to \mC$ be a homotopy invariant and Morita invariant functor. In a series of lemmas we show that $\mF$ factorises as claimed above and that the factorisation $\overline{\mF}$ is unique. 

The basic idea of the proof is that every Morita morphism can be written as the composition of an ordinary homomorphism and the inverse of an ordinary homomorphism which is also invertible under $\mF$. Establishing this fact settles uniqueness of the factorisation and essentially also existence. 

For technical reasons, we do not start with the analysis of general Morita cycles right away but consider Morita equivalences first. This helps us to see that, for every Banach algebra $A$ and every $k\in \N$, the subalgebra $A^k$ of $A$ satisfies $\mF(A^k) \cong \mF(A)$, canonically. If one only wants to treat non-degenerate Banach algebras, one can attack general Morita morphisms right from the start.

\begin{definition}
Let $E$ be a Morita equivalence between $A$ and $B$ with linking algebra $L$. Let $\iota_A\colon A \to L$ and $\iota_B\colon B \to L$ be the canonical inclusions of $A$ and $B$ into the linking algebra $L$, respectively. Define
$$
\overline{\mF}(E):= \mF(\iota_A)\cdot \mF(\iota_B)^{-1} \quad \in \quad \mC(\mF(A),\mF(B)).
$$
\end{definition}

\begin{lemma}\label{Lemma:ConcurrentHomomorphismOfMoritaEquivalences}
Let ${_{\chi}} \Phi_\psi$ be a concurrent homomorphism between Morita equivalences ${_A}E_B$ and ${_{A'}}E'_{B'}$. Then the following diagram commutes in $\mC$
$$
\xymatrix{
\mF(A) \ar[rr]^{\overline{\mF}(E)} \ar[d]_{\mF(\chi)} && \mF(B) \ar[d]^{\mF(\psi)}\\
\mF(A') \ar[rr]^{\overline{\mF}(E')} && \mF(B')
}
$$
\end{lemma}
\begin{proof}
Let $L$ be the linking algebra of $E$ and let $L'$ be the linking algebra of $L'$. The conditions on $\Phi$ imply that there is a canonical homomorphism $\Lambda \colon L \to L'$ making the following diagram in $\BanAlg$ commutative:
$$
\xymatrix{
A \ar@{^{(}->}[r] \ar[d]_{\chi} & L\ar[d]_{\Lambda} & B \ar@{_{(}->}[l] \ar[d]^{\psi} \\
A' \ar@{^{(}->}[r]   & L' & B' \ar@{_{(}->}[l] 
}
$$
Pushing this diagram to $\mC$ with the functor $\mF$, the horizontal arrows all become isomorphisms, and the exterior (commuting) square happens to be the square the commutativity of which is asserted by the lemma.
\end{proof}

\begin{corollary}
Let $\mF$ be a homotopy invariant and Morita invariant functor on $\BanAlg$. Let $k\in \N$. Then the inclusion of $A^k$ into $A$ induces an isomorphism $\mF(A^k) \cong \mF(A)$.
\end{corollary}
\begin{proof}
Without loss of generality we assume $k\geq 2$. Consider the Morita equivalence $(A,A^{k-1})$ from $A^k$ to $A$. It maps along a concurrent homomorphism to the Morita equivalence $(A,A)$ between $A$ and $A$. The preceding lemma applied to this concurrent homomorphism gives the desired result by Lemma~\ref{Lemma:ConcurrentEquivalences}.
\end{proof}

\begin{corollary}\label{Corollary:MoritaEquivalenceAandCompacts}
Let $E$ be a Morita equivalence between $A$ and $B$. Then $\mF(A) \cong \mF(\Komp_B(E))$, canonically. 
\end{corollary}
\begin{proof}
Let $\varphi$ denote the canonical map from $A$ to $\Lin_B(E)$. Let $A'$ denote the closure of the image of $\varphi$. Note that $\Komp_B(E) \subseteq A'$. Find $k\in \N$ such that $A^k\subseteq {_A}\langle E^>, E^<\rangle$. Note that $\varphi(A^k) \subseteq \Komp_B(E)$, so $(A')^k\subseteq \Komp_B(E)$. It follows that the inclusion of $\Komp_B(E)$ into $A'$ is an isomorphism under $\mF$.

Moreover, there is an obvious $A'$-valued inner product on $(E^>, E^<)$ which turns $E$ into a Morita equivalence between $A'$ and $B$, let's call it $E'$. Note that there is a concurrent homomorphism ${_\varphi}\id_{\id}$ from ${_A}E_B$ to ${_{A'}}E'_B$, so $\overline{\mF(E)} = \overline{\mF}(E') \circ \mF(\varphi)$ by Lemma~\ref{Lemma:ConcurrentEquivalences}. Since $\overline{\mF}(E)$ and $\overline{\mF}(E')$ are isomorphisms, so is $\mF(\varphi)$. Hence $\mF(A)$ is isomorphic to $\mF(A')$ which is, in turn, isomorphic to $\mF(\Komp_B(E))$. 

\end{proof}

We are now considering the general case of a Morita morphism $(F,\varphi) \in \Mban(A,B)$. There are several ways to construct an image under $\mF$ of $(F,\varphi)$  in $\mC(\mF(A), \mF(B))$. The construction we suggest here first replaces $A$ with $A^k$ for a $k\in \N$ such that $\varphi(A^k)\subseteq \Komp(F)$. Then the image under $\mF$ of the map induced by $\varphi$ from  $A^k$ to
$$
L=\left(\begin{matrix}\Komp(F) & F^>\\ F^< & B\end{matrix}\right)
$$
composed with the inverse of the image under $\mF$ of the inclusion $\iota_B$ of $B$ into $L$ gives the desired morphism in $\mC(\mF(A), \mF(B))$. But there is a subtlety that has to be settled before we proceed: The algebra $L$ is not a linking algebra between $A$ and $B$, so it is not entirely clear that $\mF (\iota_B)$ is an isomorphism. Hence the following lemma:

\begin{lemma}\label{Lemma:IncusionsIntoNonLinkingAlgebras}  Let $B$ be a Banach algebra and let $F$ be any Banach $B$-pair. Then the inclusion $\iota_B$ of $B$ into the Banach algebra
$$
L=\left(\begin{matrix}\Komp(F) & F^>\\ F^< & B\end{matrix}\right)
$$
induces an isomorphism $\mF\left(\iota_B\right)\colon \mF(B) \cong \mF(L)$. 
\end{lemma}
\begin{proof}
Note that $L$ is Morita equivalent to $B$ and a Morita equivalence is given by $E \oplus B$. The linking algebra $L'$ of this Morita equivalence is given by 
$$
L'=\left(\begin{matrix}\Komp(F) & F^> & F^>\\ F^< & B & B\\ E^< & B & B\end{matrix}\right)
$$
where the canonical inclusion $\iota'_B$  of $B$ into $L'$ is the inclusion of $B$ as the lower right-hand corner. The inclusion $\iota'_L$ of $L$ into $L'$ is the inclusion as the upper left-hand 2-by-2-matrices. Now $\iota'_L \circ \iota_B$ is the inclusion of $B$ as the middle entry of $L'$. Note that $\iota'_B$ and $\iota'_L \circ \iota_B$ are homotopic by a standard rotation, so $\mF(\iota'_B) = \mF(\iota'_L \circ \iota_B) = \mF(\iota'_L) \circ \mF(\iota_B)$ by homotopy invariance of $\mF$. Now $\mF(\iota'_B)$ and $\mF(\iota'_L)$ are isomorphisms by Morita invariance of $\mF$, so $\mF(\iota_B)$ is an isomorphism, too. 
\end{proof}

So there is no obstacle left for the following definition:

\begin{definition}\label{Definition:LiftOfFtoMoritaban} Let $A$ and $B$ be Banach algebras and $(F,\varphi) \in \Mban(A,B)$. Let $\iota_B$ and $\iota_{\Komp(F)}$ be the inclusions of $B$ and $\Komp(F)$ into the Banach algebra
$$
L=\left(\begin{matrix}\Komp(F) & F^>\\ F^< & B\end{matrix}\right),
$$
respectively. Let $\iota_{A^k}$ denote the inclusion of $A^k$ into $A$, where $k \in \N$ is such $\varphi(A^k)\subseteq \Komp(F)$. Define 
$$
\overline{\mF}(F,\varphi):= \mF(\iota_B)^{-1} \circ \mF(\iota_{\Komp(F)} \circ \varphi\restr_{A^k}) \circ \mF(\iota_{A^k})^{-1}  
$$

\end{definition}

%

\begin{lemma}\label{Lemma:HomotopyMoritabanAndF}
Let $(F,\varphi)$ be a homotopy of Morita cycles in $\Mban(A,B)$ and let $\ev_t^B\colon B[0,1] \to B$ denote the evaluation at $t\in [0,1]$. Then 
$$
\overline{\mF}(F_t,\varphi_t) = \mF(\ev_t^B) \circ \overline{\mF}(F, \varphi).
$$
In particular, the above construction is invariant under homotopy. 
\end{lemma}
\begin{proof}
Let $t\in [0,1]$. Consider the following diagram, where the central vertical arrow is the obvious evaluation homomorphism in each entry:
$$
\xymatrix{
A \ar@{=}[d] & A^k \ar[l] \ar[r] \ar@{=}[d]& {\left(\begin{matrix}\Komp_{B[0,1]}(F) & F^>\\ F^< & B[0,1] \end{matrix}\right)}\ar[d] & B[0,1] \ar[d]_{\ev_t^B} \ar[l]\\
A & A^k \ar[l] \ar[r]& {\left(\begin{matrix}\Komp_B(F_t) & F_t^>\\ F_t ^< & B \end{matrix}\right)} & B\ar[l]
}
$$
This diagram is clearly commutative an all arrows pointing to the left are isomorphisms after one has applied $\mF$ to them. If you do apply $\mF$, then going from $A$ (or rather $\mF(A)$) in the upper left corner first right and then down, you end up with $\mF(\ev_t^B) \circ \overline{\mF}(F, \varphi)$. If you go first down and the right, you and up with $\overline{\mF}(F_t,\varphi_t)$.
\end{proof}

\begin{lemma}\label{Lemma:FonMoritabanMultiplicative}
Let $A$, $B$ and $C$ be Banach algebras and let $(F, \varphi) \in \Mban(A,B)$ and $(F', \varphi') \in \Mban(B,C)$. Then
$$
\overline{\mF} (F\otimes_B F') = \overline{\mF}(F') \circ \overline{\mF}(F) \ \in \ \mC(\mF(A),\mF(C))
$$
\end{lemma}
\begin{proof}
We first give a proof under the additional assumptions that $A$, $B$ and $C$ are non-degenerate and that $F$ is a non-degenerate $B$-pair, i.e., $F^>B = F^>$ and $BF^< = F^<$. In this situation, $\varphi(A) \subseteq \Komp_B(F)$ and $\varphi'(B) \subseteq \Komp_C(F')$ as well as $\Komp_B(F) \otimes \id_{F'} \subseteq \Komp_C(F \otimes_B F')$. Consider the following diagram:
$$
\xymatrix{
&{\left(\begin{matrix}\Komp_B(F) & F^>\\ F^< & B \end{matrix}\right)}\ar[dd]_{\alpha}&\\
A \ar[dd]\ar[ru]  & & B \ar[ul] \ar[dd] \\
&D&\\
    {\left(\begin{matrix}\Komp_C(F\otimes_B F') & F^> \otimes_B F'^>\\ F'^< \otimes_B F^< & C \end{matrix}\right)}\ar[ru]^{\gamma} &  &  \qquad {\left(\begin{matrix}\Komp_C(F') & F'^>\\ F'^< & C \end{matrix}\right)}\qquad \ar[lu]_{\beta}\\
    & C\ar[ur] \ar[ul] &
}
$$
The morphisms appearing in the exterior hexagon are those that are used in the definition of $\overline{\mF}$; it hence suffices to show that this hexagon is commutative after applying $\mF$ to all morphisms (a process under which some of the morphisms become isomorphisms so that it makes actually sense to speak of commutativity of the diagram). The Banach algebra $D$ in the center is defined as
$$
D:= \Komp_C(F \otimes_C F' \oplus F' \oplus C)
$$
so it can be thought of as an algebra of 3-by-3 matrices. The map $\alpha$ takes ${\left(\begin{matrix}\Komp_B(F) & F^>\\ F^< & B \end{matrix}\right)}$ first to $\Komp_C(F \otimes_B F' \oplus F') = {\left(\begin{matrix}\Komp_C(F\otimes_B F') & \Komp_C(C,F')\\ \Komp_C(F',C)& \Komp_C(C) \end{matrix}\right)}$ and then to $D$ as upper-left 2-by-2 block matrices. The map $\beta$ is defined similarly, with the difference that we end up in the lower-right block of 2-by-2 matrices in $D$.Note that $\beta$ maps to an isomorphism under $\mF$. Finally, the map $\gamma$ is again defined in a similary fashion, but this time, the inclusion is contained in the set of matrices that have vanishing second row and vanishing second column. The internal squares of the diagram commute, hence we are done.

The basic idea of the proof in the general case is very similar, but the details are somewhat more tiresome, so we do not give them here; the crucial point is that you have to replace the algebras in the vertices of the hexagon by sufficiently high powers of themselves to give a meaning to all arrows.
\end{proof}

The following lemma is elementary to check so we omit the proof; the key ingredient is a homotopy between the inclusions of $BB$ into $\Mat_2(BB)$ into the upper left and lower right corners.

\begin{lemma}\label{Lemma:OverlineFLiftsF}
Let $\varphi \colon A \to B$ be a homomorphism of Banach algebras. Then $\overline{\mF}(\Mban(\varphi)) = \mF(\varphi)$.
\end{lemma}

To sum up, we have shown that the map $\overline{\mF}\colon \Moritabancat \to \mC$ defined in \ref{Definition:LiftOfFtoMoritaban} is well-defined (Lemma \ref{Lemma:HomotopyMoritabanAndF}) and a functor (Lemma \ref{Lemma:FonMoritabanMultiplicative} and Lemma \ref{Lemma:OverlineFLiftsF}). It lifts $\mF$ by Lemma  \ref{Lemma:OverlineFLiftsF}, so we have shown Theorem \ref{Theorem:LiftToMoritaban}.

We are now able to prove Morita invariance of $\KTh$-theory also for possibly degenerate Banach algebras; note that it only uses the less complicated part of Theorem \ref{Theorem:LiftToMoritaban}.

\begin{theorem}\label{Theorem:KTheoryMoritaInvariant}
The $\KTh$-functor for Banach algebras is Morita invariant 
\end{theorem}
\begin{proof}
The basic idea of how to prove this result is to go through the article \cite{Paravicini:07:Morita:richtigerschienen} and to generalise everything to possibly degenerate Banach algebras. Instead of actually carrying out this plan, I'll just give the list of the necessary changes in  \cite{Paravicini:07:Morita:richtigerschienen} and also in \cite{Lafforgue:02}:
\begin{itemize}
\item First, we have to adjust the definition of $\KKban$ to allow for possibly degenerate Banach algebras. This is done analogously to the case of Morita cycles discussed above and just as sketched in the Remarque following Th\'{e}or\`{e}me 1.2.8 of \cite{Lafforgue:02}: cycles in $\Eban(A,B)$ consist of pairs $(E,T)$ where $E$ is a (possibly degenerate) graded Banach $B$-pair and $T$ is an odd element of $\Lin_B(E)$ such that there is a $k\in \N$ such that $a (T^2-1)$ and $[a,T]$ is in $\Komp_B(E)$ for all $a\in A^k$. The main change concerns the homotopies: A homotopy between elements of $\Eban(A,B)$ is an element $(E,T)$ in $\Eban(A, B[0,1])$ such that $E$ is not only a Banach $B[0,1]$-pair but a $\Cont[0,1]$-Banach $B$-pair in the sense of \cite{Paravicini:10:GreenJulg:submitted}, see also Definition~\ref{Definition:HomotopyOfMoritaCycles} above. The operator $T$ has to be $\unital{B}[0,1]$-linear. This means that homotopies are fibred over $[0,1]$, and you can then view a homotopy as connecting the fibres over $0$ and $1$.
\item Now you have to show that $\KKban(\C, B) \cong \KTh_0(B)$, also for degenerate Banach algebras $B$. To this end, you go through the relevant parts of the first chapter of \cite{Lafforgue:02} to reprove Th\'{e}or\`{e}me 1.2.8 in this more general setting. Some changes have to be made: For example, the canonical map from $B$ to $\Lin_B(B)$ is no longer $\Komp_B(B)$-valued. This can be repaired by considering $BB$ instead of $B$. Note that $B$ and $BB$ have the same $\KTh$-theory because of the six-term exact sequence in $\KTh$-theory (the quotient of $B$ by $BB$ is contractible). You have to adjust Lemme 1.1.6 and Lemme 1.1.7 accordingly. Also, in Lemme 1.3.5, you have to consider elements $y \in \Mat_n(B)\Mat_n(B)$ instead of $y\in \Mat_n(B)$.
\item The action of Morita morphisms on $\KKban$ is defined as in Section 5.7 of \cite{Paravicini:07:Morita:richtigerschienen}. In particular, $\KKban(\C, \cdot)$ factors through $\Moritaban$. Now apply Theorem \ref{Theorem:LiftToMoritaban}.
\end{itemize}
\end{proof}

\section{Exactness and Bott-periodicity for functors on $\BanAlg$}

\subsection{Half-exactness and split-exactness}

\begin{definition}
Let $\mC$ be a category and $\mF \colon \BanAlg \to \mC$ a functor.
\begin{itemize}
\item If $\mC$ is an abelian category, then $\mF$ is called \demph{split-exact} if, for every extension $\xymatrix{B \ \ar@{>->}[r]& D \ar@{->>}[r]^-{\pi}& A}$ of Banach algebras such that the quotient map has a split $\sigma$ that is a morphism of Banach algebras, the sequence $\xymatrix{\mF(B) \ \ar@{>->}[r]& \mF(D) \ar@{->>}[r]^-{\mF(\pi)}& \mF(A)}$ is a short exact in $\mC$ with split $\mF(\sigma)$. 
\item If $\mC$ is an abelian category and if $\mE$ denotes a class of extensions of Banach algebras, then $\mF$ is called \demph{half-exact for $\mE$} if, for every extension $\xymatrix{B \ \ar@{>->}[r]& D \ar@{->>}[r]& A}$ in $\mE$, the sequence $\xymatrix{\mF(B) \ \ar@{->}[r]& \mF(D) \ar@{->}[r]& \mF(A)}$ in $\mC$ is exact in the middle.
\item If $\mC$ is (just) an additive category, then $\mF$ is called \demph{split-exact} if, for every object $X$ of $\mC$, the functors $A \mapsto \Hom(X, \mF(A))$ from $\BanAlg$ to $\Ab$ and $A\mapsto \Hom(\mF(A), X)$ from $\BanAlg$ to $\Ab^{\op}$ are split-exact. Similarly, one defines \demph{half-exactness} for functors with values in an additive category. 
\end{itemize}
\end{definition}

\begin{proposition}[{\it cf.}~\cite{Schochet:84}, Prop.~2.4, \cite{CuntzThom:06}, Lemma~4.1.5, and  \cite{CMR:07}, Prop.~6.71]\label{Proposition:HalfexactnessLongExactSequence}
Let $\mF$ be a functor on $\BanAlg$ with values in an abelian category that is half-exact for all (semi-split) extensions and homotopy invariant. Define
$$
\mF_k(A):= \mF(\Sigma^k A)
$$
for all Banach algebras $A$ and all $k\geq 0$. Then the functor $\mF$ has long exact sequences of the form
$$
\xymatrix{
\cdots \ar[r] & \mF_1(B) \ar[r] & \mF_1(D) \ar[r]& \mF_1(A) \ar[r]& \mF_0(B) \ar[r]& \mF_0(D) \ar[r]& \mF_0(A)
}
$$
for any (semi-split) extension $\xymatrix{B \ \ar@{>->}[r]& D \ar@{->>}[r]^{\pi}& A}$; the injection $\kappa_{\pi}\colon B \to \cone_{\pi}, b\mapsto (b,0)$ induces an isomorphism $\mF_k(B) \cong \mF_k(\cone_{\pi})$ and the connecting map of the above exact sequence is given by $\xymatrix{\mF_k(\Sigma A) \ar[r]^{\mF_k(\iota(\pi))} & \mF_k(C_{\pi}) & \mF_k(B) \ar[l]_-{\mF_k(\kappa_{\pi})}^{\cong}}$, for all $k\geq 0$. The functor $\mF$ is split-exact.
\end{proposition}

\begin{corollary}\label{Corollary:KappaIsomorphism:Additive}
Let $\mF$ be a functor on $\BanAlg$ with values in an \emph{additive} category that is half-exact for all (semi-split) extensions and homotopy invariant. Then, for any (semi-split) extension $\xymatrix{B \ \ar@{>->}[r]& D \ar@{->>}[r]^{\pi}& A}$, the injection $\kappa_{\pi}\colon B \to \cone_{\pi}, b\mapsto (b,0)$ induces an isomorphism $\mF(B) \cong \mF(\cone_{\pi})$.
\end{corollary}
\begin{proof}
Use Proposition \ref{Proposition:HalfexactnessLongExactSequence} on the functors $A \mapsto \Hom(X, \mF(A))$ and $A \mapsto \Hom(\mF(A), X)$; this gives that composition with $\mF(\kappa_\pi)$ induces isomorphisms $\Hom(X, \mF(B)) \cong \Hom(X, \mF(\cone_{\pi}))$ and $\Hom(\mF(\cone_{\pi}), X) \cong \Hom(\mF(B), X))$. Evaluate this on $X=\mF(\cone_{\pi})$ and $X=\mF(B)$, respectively, to see that $\mF(\kappa_{\pi})$ is invertible from the left and from the right.
\end{proof}

\begin{corollary}\label{Corollary:HalfexactnessLongExactSequence:Additive}
Let $\mF$ be a functor on $\BanAlg$ with values in an additive category that is half-exact for semi-split extensions and homotopy invariant. Then $\mF$ is split-exact.
\end{corollary}

\begin{lemma}\label{Lemma:HalfExactTwoStabilisation} Let $\mF$ be a functor on $\BanAlg$ with values in an additive category $\mC$ that is half-exact for semi-split extensions and homotopy invariant. Then, for every Banach algebra $A$, the canonical map from $\Sigma \otimes A$ to $\Sigma A$ maps to an isomorphism $\mF(\Sigma \otimes A) \cong \mF(\Sigma A)$. 
\end{lemma}
\begin{proof}
We first consider the case that $\mC$ is abelian. 

The semi-split extension 
$$
\xymatrix{ 
\Sigma \otimes A\ \ar@{>->}[r]& \cone \otimes A \ar@{->>}[r]& A
}
$$
has a contractible algebra in the middle; so, by Proposition \ref{Proposition:HalfexactnessLongExactSequence}, it gives a long exact sequence:
$$
\xymatrix{
\cdots \ar[r] & 0 \ar[r]& \mF_1(A) \ar[r]& \mF_0(\Sigma \otimes A) \ar[r]& 0 \ar[r]& \mF_0(A)
}
$$
which yields an isomorphism $\mF(\Sigma A)\cong \mF(\Sigma \otimes A)$. Now here is a simple way to check the more precise result that the canonical map from $\Sigma \otimes A$ to $\Sigma A$ is mapped to an isomorphism under $\mF$: There is a commutative diagram
$$
\xymatrix{ 
\Sigma \otimes A\ \ar@{>->}[r] \ar[d]& \cone \otimes A \ar@{->>}[r]\ar[d]& A\ar@{=}[d]\\
\Sigma A\ \ar@{>->}[r]& \cone  A \ar@{->>}[r]& A
}
$$
which, after taking long exact sequences, gives the diagram
$$
\xymatrix{
\cdots \ar[r] & 0\ar@{=}[d] \ar[r]& \mF_1(A) \ar@{=}[d] \ar[r]^-{\cong}& \mF_0(\Sigma \otimes A) \ar[r] \ar[d]& 0 \ar[r] \ar@{=}[d]& \mF_0(A)\ar@{=}[d]\\
\cdots \ar[r] & 0 \ar[r]& \mF_1(A) \ar[r]^-{\cong}& \mF_0(\Sigma A) \ar[r]& 0 \ar[r]& \mF_0(A)
}
$$
This implies the claim for $\mC$ abelian. To prove it for general additive categories consider the functors $A \mapsto \Hom(X, \mF(A))$ and $A\mapsto \Hom(\mF(A), X)$ for $X=\mF(\Sigma A)$ and $X=\mF(\Sigma \otimes A)$.
\end{proof}

\begin{definition}[Homology theory for Banach algebras, {\it cf.}~p.~400 in \cite{Schochet:84} and Def.~6.68 in \cite{CMR:07}]\label{Definition:HomologyTheory} A \demph{homology theory for Banach algebras} is a sequence of covariant functors $(\mF_k)_{k\in \Z}$ from $\BanAlg$ to an abelian category together with natural isomorphisms $\mF_k(\Sigma A) \cong \mF_{k+1}(A)$ for all $k\in \Z$, such that
\begin{enumerate}
\item the functors $\mF_k$ are homotopy invariant;
\item the functors $\mF_k$ are half-exact for semi-split extensions.
\end{enumerate}
\end{definition}

\noindent We do not exactly need the following result on homology theories, but it is somewhat reassuring to know that homology theories generously ignore at least some of the technical subtleties that one encounters in the world of Banach algebras.

Let $\CW_2$ denote the category of pairs of finite CW-complexes. There are two seemingly different ways of defining an ``action'' of $\CW_2$ on $\BanAlg$: If $A$ is a Banach algebra and $(X,X_0)$ is a pair of finite CW-complexes, then one can either form the Banach algebra $\Cont((X,X_0), A) \cong \Cont_0(X-X_0, A)$ or the Banach algebra $\Cont_0(X-X_0) \otimes A$, where the tensor product is the completed projective tensor product. One can compare the two algebras using the canonical homomorphism from $\Cont_0(X-X_0) \otimes A$ to $\Cont_0(X-X_0, A)$, but note that it is rarely an isomorphism of Banach algebras. However, we have the following result:

\begin{proposition}\label{Proposition:ActionOfSpacesOnFunctors}
Let $(\mF_k)_{k\in \Z}$ be a homology theory for Banach algebras. Let $A$ be a Banach algebra and $(X,X_0)$ an object in $\CW_2$. The natural homomorphism from $\Cont_0(X-X_0) \otimes A$ to $\Cont_0(X-X_0, A)$ induces an isomorphism $\mF_k(\Cont_0(X-X_0) \otimes A) \cong \mF_k(\Cont_0(X-X_0, A))$ for all $k\in \Z$.
\end{proposition}
\begin{proof}
Note that, if we fix $A$, we are comparing two (generalised) homology theories on $\CW_2$, namely $(\mF_k(\Cont_0(X-X_0) \otimes A))_{k\in \Z}$ and $(\mF_k(\Cont_0(X-X_0, A)))_{k\in \Z}$ which are connected by a natural transformation. The arguments that this natural transformation is indeed a natural equivalence is somewhat classical, but we repeat it for the reader's convenience. We concentrate on the equivalence for $k=0$ and define $\mF:=\mF_0$.
\begin{itemize}
\item It follows from Lemma \ref{Lemma:HalfExactTwoStabilisation} that $\mF(\Sigma \otimes A) \cong \mF(\Sigma A)$, canonically.
\item If $n\in \N$, then $\mF(\Sigma^{n}A) = \mF(\Sigma(\Sigma^{n-1} A)) \cong \mF(\Sigma \otimes (\Sigma^{n-1} A))$. So, by induction, $\mF(\Sigma \otimes (\Sigma^{n-1} A)) \cong \mF(\Sigma \otimes (\Sigma^{n-1} \otimes A)) \cong \mF((\Sigma \otimes \Sigma^{n-1}) \otimes A) \cong \mF(\Sigma^{n} \otimes A)$. All in all, we have
$$
\mF(\Sigma^{n} \otimes A)\cong \mF(\Sigma^{n}A),
$$
canonically and naturally. In other words, if $\mathring{D}^n$ denotes the open $n$-dimensional unit disk, then $\mF(\Cont_0(\mathring{D}^n) \otimes A) \cong \mF(\Cont_0(\mathring{D}^n, A))$. 

\item The $n$-dimensional unit sphere is the one-point compactification of $\mathring{D}^n$. We hence get a diagram of split extensions
$$
\xymatrix{ 
\Cont_0(\mathring{D}^n) \otimes A\ \ar@{>->}[r] \ar[d] & \Cont(S^n) \otimes A \ar@{->>}[r]\ar[d]& A \ar@{=}[d]\\
\Cont_0(\mathring{D}^n, A) \ \ar@{>->}[r]& \Cont(S^n, A) \ar@{->>}[r]& A
}
$$
Two of the vertical arrows are isomorphisms after applying $\mF$, so the same is true for the third. In other words, we have a natural isomorphism
$\mF(\Cont(S^n) \otimes A) \cong \mF(\Cont(S^n, A))$. 

\item Similarly, one shows that $\mF(\Cont(D^n) \otimes A) \cong \mF(\Cont(D^n, A))$; you can also use homotopy invariance to arrive at this fact.

\item Let $X$ be a finite CW-complex; for all $n\in \N_0$, let $X^n$ denote its $n$-skeleton. We have a push-out square
$$
\xymatrix{
X^n & X^{n-1} \ar[l]\\
\bigcup_i D^n \ar[u]&  \bigcup_i S^{n-1}\ar[u] \ar[l]
}
$$

where the hoizontal arrows are embeddings. It follows that we have a push-out square 
$$
\xymatrix{
\Cont(X^n) \ar[r]\ar[d] & \Cont(X^{n-1}) \ar[d]\\
\Cont(\bigcup_i D^n) \ar[r]&  \Cont(\bigcup_i S^{n-1}) 
}
$$
where the horizontal maps are surjective. Because the subcomplex $X^{n-1}$ is a strong neighbourhood retract in $X^n$ one can extend this square to the following diagram where the rows are semi-split extensions:
$$
\xymatrix{
\Cont_0(\bigcup_i \mathring{D}^n) \ar@{>->}[r] \ar@{=}[d] & \Cont(X^n) \ar@{->>}[r]\ar[d] & \Cont(X^{n-1}) \ar[d]\\
\Cont_0(\bigcup_i \mathring{D}^n) \ar@{>->}[r]& \Cont(\bigcup_i D^n) \ar@{->>}[r]&  \Cont(\bigcup_i S^{n-1}) 
}
$$
From the first line of this diagram and the corresponding diagram with coefficients in $A$ we obtain the following commutative diagram
$$
\xymatrix{
\Cont_0(\bigcup_i \mathring{D}^n) \otimes A \ar@{>->}[r] \ar[d] & \Cont(X^n)\otimes A \ar@{->>}[r]\ar[d] & \Cont(X^{n-1}) \otimes A\ar[d]\\
\Cont_0(\bigcup_i \mathring{D}^n,A ) \ar@{>->}[r]& \Cont(X^n,A) \ar@{->>}[r]& \Cont(X^{n-1},A) 
}
$$
The left vertical arrow is an isomorphism after applying $\mF$ by what we have said above (recall that $X$ was supposed to be a \emph{finite} CW-complex). If we proceed by induction on $n$, starting with $n=0$, we can deduce that the right vertical arrow is an isomorphism after applying $\mF$. Hence also the vertical arrow in the centre is an isomorphism after applying $\mF$ (here we use a five-lemma argument which needs the long-exact sequences for $\mF$ also in negative degrees). 

We have hence shown: If $X$ is a finite CW-complex, then there is a natural isomorphism $\mF(\Cont(X) \otimes A) \cong \mF(\Cont(X,A))$.

\item Let $(X,X_0)$ be a pair of finite CW-complexes. Then we obtain a diagram
$$
\xymatrix{
\Cont_0(X-X_0) \otimes A \ar@{>->}[r] \ar[d] & \Cont(X)\otimes A \ar@{->>}[r]\ar[d] & \Cont(X_0) \otimes A\ar[d]\\
\Cont_0(X-X_0,A ) \ar@{>->}[r]& \Cont(X,A) \ar@{->>}[r]& \Cont(X_0,A) 
}
$$
Again, the subcomplex $X_0$ is a strong neighbourhood retract of $X$, so we can find continuous linear splitings of both extensions. We can hence deduce that there is a natural isomorphism: 
$$
\mF(\Cont_0(X-X_0) \otimes A) \cong \mF(\Cont_0(X-X_0,A)).
$$

\end{itemize}

\end{proof}

\subsection{Double split extensions and quasi-homomorphisms}

\begin{definition}[Double split extension,  {\it cf.}~\cite{CMR:07, Cuntz:87}]
A \demph{double split extension} is an extension $\xymatrix{B \ \ar@{>->}[r]& D \ar@{->>}[r]^{\pi}& A}$ of Banach algebras equipped with two continuous homomorphisms $\varphi_+$ and $\varphi_{-}$ from $A$ to $D$ which are both splits of $\pi$. 
\end{definition}

Let $\mF$ be a split-exact functor on the Banach algebras with values in some additive category $\mC$.

Let $\xymatrix{B \ \ar@{>->}[r]& D \ar@{->>}[r]^{\pi}& A}$ be a double split extension with splits $\varphi_{\pm}$. We obtain a split exact sequence
$$
\xymatrix{\mF(B) \ \ar@{>->}[r]& \mF(D) \ar@{->>}[r]& \mF(A)}
$$
with two sections $\mF(\varphi_{+})$ and $\mF(\varphi_{-})$. Thus we get a map
$$
\mF(\varphi_{\pm})\colon \mF(\varphi_{+}) - \mF(\varphi_{-}) \colon \mF(A) \to \mF(B) \subseteq \mF(D). 
$$
So every double split extension of Banach algebras induces a morphism on the level of $\mF$.

\begin{definition}[Generalised ideal, {\it cf.}~\cite{CMR:07}, Def.~3.1]
Let $B$ and $D$ be Banach algebras and let $\iota \colon B \to D$ be an injective bounded homomorphism. We call $B$ a (generalised) ideal in $D$ if the multiplication on $D$ restricts to bounded linear maps $B \times D \to B$ and $D \times B \to B$; here we identify $B$ and $\iota(B)$. The ideal is called closed if $\iota(B)$ is closed in $D$, i.e., if $\iota$ is an isomorphism onto its image. 
\end{definition}

\begin{definition}[Quasi-homomorphism, {\it cf.}~\cite{CMR:07}, Def.~3.2 and \cite{Cuntz:87}]
Let $A$, $B$ and $D$ be Banach algebras and suppose that $B$ is a generalised ideal in $D$. A quasi-homomorphism $A \rightrightarrows D \vartriangleright B$ is a pair of bounded homomorphisms $\varphi_{\pm}\colon A\to D$ such that $\varphi_{+}(a) - \varphi_{-}(a) \in B$ for all $a\in A$ \fxnote{reicht hier $\forall a \in A^k$ fuer ein $k\in \N$?} and the resulting linear map $\varphi_{+}-\varphi_{-}\colon A \to B$ is bounded. It is called special if the map $A \oplus B \to D, (a,b) \mapsto \varphi_{+}(a) + b$ is bijective (and hence an isomorphism).
\end{definition}

Let $A \rightrightarrows D \vartriangleright B$ be a special quasi-homomorphism. Then $B$ is a closed ideal in $D$ and $D/B \cong A$ via $\varphi_{\pm}^{-1}$, so that we obtain an extension of Banach algebras $\xymatrix{B \ \ar@{>->}[r]& D \ar@{->>}[r]& A}$. The bounded homomorphisms $\varphi_{+}$ and $\varphi_{-}$ are sections for this extension. If $\mF$ is as above, we hence get an element
$$
\mF(\varphi_{\pm})\colon \mF(\varphi_{\pm}) - \mF(\varphi_{-}) \colon \mF(A) \to \mF(B) \subseteq \mF(D). 
$$
induced by the given special quasi-homomorphism. 

If the quasi-homomorphism is not special, proceed as follows ({\it cf.}~p.~47 of \cite{CMR:07}). Define $D'$ to be the Banach space $B \oplus A$, equipped with the multiplication
$$
(b_1, a_1) \cdot (b_2, a_2) := (b_1 b_2 + \varphi_+(a_1) b_2 +b_1 \varphi_+(a_1), \cdot a_1 a_2);
$$
here, we regard $B$ as a subset of $D$. It is easy to check that this is bounded and associative. We obtain an extension of Banach algebras
$$
\xymatrix{B \ \ar@{>->}[r]& D' \ar@{->>}[r]& A}
$$
where the homomorphisms are given by $b \mapsto (b, 0)$ and $(b,a) \mapsto a$. The extension has two sections, namely
$$
\varphi_{\pm}' \colon A \to D', \quad \varphi_+'(a) := (0,a), \quad \varphi_-'(a) := (\varphi_-(a) - \varphi_+(a), a).
$$
The maps $\varphi_+'$ and $\varphi_-'$ are bounded homomorphisms and form a special quasi-homomorphism $\varphi'_{\pm}\colon A \rightrightarrows D' \vartriangleright B$.

The following proposition is a transcription of Proposition 3.3 of \cite{CMR:07} which is formulated in the setting of bornological algebras. We include it here for the convenience of the reader but omit the proof which could also be copied verbatim.

\begin{proposition}\label{Proposition:EigenschaftenDoppelspalt}
The construction of $\mF(\varphi_{\pm})$ has the following properties:
\begin{enumerate}
\item Consider a commuting diagram
$$
\xymatrix{
A \ar@<3pt>[r]^{\varphi_+} \ar@<-3pt>[r]_{\varphi_-} & D_1 \ar@{}[r]|{\vartriangleright} \ar[d]^{\psi_D}& B_1 \ar[d]^{\psi_B}\\
 & D_2 \ar@{}[r]|{\vartriangleright}& B_2
}
$$
whose first row is a quasi-homomorphism. Then $(\psi_D \circ \varphi_{\pm}) \colon A \rightrightarrows D_2 \vartriangleright B_2$ is a quasi-homomorphism, and $\mF(\psi_D \circ \varphi_{\pm}) = \mF(\psi_B) \circ \mF(\varphi_{\pm})$. 
\item We have 
$$
\mF(\varphi, \varphi) = 0.
$$
\item If $(\varphi_+, \varphi_-)$ is a quasi-homomorphism, then so is $(\varphi_, \varphi_+)$, and 
$$
\mF(\varphi_+, \varphi_-) = -\mF(\varphi_-, \varphi_+).
$$
\item If $(\varphi_+, \varphi_-)$ and $(\varphi_+, \varphi_0)$ are quasi-homomorphism, then so is $(\varphi_, \varphi_0)$, and 
$$
\mF(\varphi_+, \varphi_0) + \mF(\varphi_0, \varphi_-) = \mF(\varphi_+, \varphi_-).
$$
\item If $\varphi_{\pm}$ is a pair of bounded homomorphisms from $A$ to $B$, then \label{Proposition:EigenschaftenDoppelspalt:EchteHoms}
$$
\mF(\varphi_{\pm}) = \mF(\varphi_+) - \mF(\varphi_-).
$$
\item Two quasi-homomorphisms $(\varphi^1_+, \varphi^1_-), (\varphi^2_+, \varphi^2_-)\colon A \rightrightarrows D \vartriangleright B$ are called \demph{orthogonal} if $\varphi^1_+(x) \varphi^2_+(y) = 0=\varphi^2_+(x) \varphi^1_+(y)$ and $\varphi^1_-(x) \varphi^2_-(y) = 0=\varphi^2_-(x) \varphi^1_-(y) $ for all $x,y\in A$. In this case, $\varphi_+^1+\varphi_+^2$ and $\varphi_-^1+\varphi_-^2$ are homomorphisms and we get a quasi-homomorphism
$$
(\varphi_+^1, \varphi_-^1) + (\varphi_+^2, \varphi_-^2):= (\varphi_+^1+\varphi_+^2, \varphi_-^1+\varphi_-^2)\colon A \rightrightarrows D \vartriangleright B.
$$
We have
$$
\mF\left((\varphi_+^1, \varphi_-^1) + (\varphi_+^2, \varphi_-^2)\right) = \mF(\varphi_+^1, \varphi_-^1) + \mF(\varphi_+^2, \varphi_-^2).
$$
\end{enumerate}
\end{proposition}

\subsection{Bott periodicity}

\begin{theorem}[{\it cf.}~\cite{Cuntz:84}] \label{Theorem:BottPeriodicity:Functor}
Let $\mF$ be a functor from $\BanAlg$ into an additive category which is half-exact for semi-split extensions, homotopy invariant and Morita invariant. Then there is a natural isomorphism $\mF(\Sigma^2 A) \cong \mF(A)$ for every Banach algebra $A$.
\end{theorem}
\begin{proof}
We give the proof on two levels of detail: a short argument to which Martin Grensing has pointed me, and my original proof, a longer and more explicit version of the argument that gives more information on why the isomorphism is natural.

Let $A$ be a Banach algebra and consider the functor $\mF(\cdot \otimes A)$. Note that this functor is half-exact (for semi-split extensions), homotopy invariant and Morita invariant (and thus split-exact by Corollary \ref{Corollary:HalfexactnessLongExactSequence:Additive}). We can restrict it to the category of C$^*$-algebras, and the restricted functor has the same properties in the C$^*$-algebraic sense. It hence satisfies Bott periodicity (see \cite{Cuntz:84}) which means in particular that $\mF(\Sigma^2 \otimes A) \cong \mF(A)$. Now $\mF(\Sigma^2 \otimes A) \cong \mF(\Sigma^2 A)$ by Lemma \ref{Lemma:HalfExactTwoStabilisation}, so we have obtained the desired isomorphism. 

More explicitly, let $\mT$ denote the Toeplitz algebra as defined in \cite{Cuntz:84}; likewise, let $\mT_0\subseteq \mT$ denote the reduced Toeplitz algebra. Note that $\mF(\mT_0\otimes A)=0$, which follows as in Proposition 4.3 of \cite{Cuntz:84}, applied to the functor $\mF(\cdot \otimes A)$ (with the slight change that half-exactness is only valid for semi-split extensions); compare also the smooth version in \cite{CMR:07}. Note that $\mF$ and therefore also $\mF(\cdot \otimes A)$  are split-exact by Corollary \ref{Corollary:HalfexactnessLongExactSequence:Additive}.

Let $A$ be a Banach algebra. Consider the reduced Toeplitz extension 
$$
\xymatrix{ \Komp(\ell^2(\N)) \ \ar@{>->}[r]& \mT_0 \ar@{->>}[r]& \Sigma}.
$$
Taking the completed projective tensor product of this (semi-split) extension with the Banach algebra $A$ gives an extension 
\begin{equation}\label{Equation:TopelitzExtension}
\xymatrix{ 
\Komp(\ell^2(\N)) \otimes A\ \ar@{>->}[r]& \mT_0 \otimes A \ar@{->>}[r]& \Sigma \otimes A.
}
\end{equation}
Note that taking the injective tensor product would be somewhat more natural but it is not clear that this would give Banach algebras. Now the canonical map from $\Sigma \otimes A$ to $\Sigma A$ induces an isomorphism $\mF(\Sigma \otimes A) \cong \mF(\Sigma A)$ by Lemma \ref{Lemma:HalfExactTwoStabilisation}.


Note also that $\Komp(\ell^2(\N)) \otimes A$ is (naturally) Morita equivalent to $A$, and hence $\mF(\Komp(\ell^2(\N)) \otimes A)\cong \mF(A)$, naturally.

Consider now the case that $\mF$ has values in an abelian category. By Proposition \ref{Proposition:HalfexactnessLongExactSequence}, the extension (\ref{Equation:TopelitzExtension}) gives a long exact sequence 
$$
\xymatrix{
\cdots \ar[r] & \mF_1( \mT_0 \otimes A) \ar[r]& \mF_1(\Sigma  A) \ar[r]& \mF (\Komp(\ell^2(\N)) \otimes A) \ar[r]& \mF (\mT_0 \otimes A) \ar[r]& \mF(\Sigma \otimes A)
}
$$
where $\mF_1(\cdot) = \mF(\Sigma\cdot)$. Using our above calculations (also for $\mF_1$ instead of $\mF$), we arrive at the following exact sequence
$$
\xymatrix{
\cdots \ar[r] & 0 \ar[r]& \mF_1(\Sigma  A) \ar[r]& \mF (A) \ar[r]& 0 \ar[r]& \mF(\Sigma A).
}
$$
So we have an isomorphism $\mF(\Sigma^2 A) = \mF_1(\Sigma A) \cong \mF(A)$, as claimed. It is clearly natural in $A$. 

To treat the case that $\mF$ takes its values only in an additive category consider the functors $A \mapsto \Hom(X, \mF(A))$ and $A\mapsto \Hom(\mF(A), X)$ for arbitrary $X$.
\end{proof}

\begin{remark}
In a first attempt to show Bott periodicity I used the following variant of the Toeplitz algebra (with coefficients in a Banach algebra $A$):
$$
\cT^{\ban}_A:= \ell^1(\N_0 \times \N_0, A) \oplus \ell^1(\Z, A) \cong \left(\ell^1(\N_0 \times \N_0) \oplus \ell^1(\Z)\right) \otimes A
$$
with a  product similar to that defined on page 64 of \cite{CMR:07}. With this algebra, one can only show that $\mF(A)$ is isomorphic to $\mF(\Sigma \Sigma^{\ban} A)$, where $\Sigma^{\ban} A= \ell^1_0(\Z, A)= \{f \in  \ell^1(\Z, A) | \sum_{k\in \Z} f(k) =0\}$. This algebra sits as a dense subalgebra in $\Sigma A$, but it is not clear that $\mF(\Sigma \Sigma^{\ban} A)\cong \mF(\Sigma^2 A)$  unless you know that $\mF$ is invariant under dense and spectral homomorphisms.
\end{remark}

The condition that $\mF$ is homotopy invariant is automatic and hence redundant because we have the following variant of Lemma~3.26 of \cite{CMR:07}, {\it cf.} \cite{Kasparov:81:Extensions, Higson:88, CuntzThom:06}. Note that the hypotheses of the result are somewhat oversimplified: much less than Morita invariance will probably do, but we will not venture into this.

\begin{proposition}
Let $\mF$ be a split-exact Morita invariant functor on $\BanAlg$. Then $\mF$ is homotopy invariant.
\end{proposition}
\begin{proof}
Let $A$ be a Banach algebra. We consider the Banach algebra $\mK^*A$ from Chapter 3 of \cite{CMR:07}:
$$
\mK^*A:= c_0(\N, \ell^1(\N, A)) \cap \ell^1(\N, c_0(\N, A)) 
$$
as a completion of $c_c(\N \times \N, A)$. You can think of this as a groupoid Banach algebra of the groupoid $\N \times \N$. In Lemma~3.26 of \cite{CMR:07} it is shown that, if $\mF$ is a functor on the category of bornological algebras that is split-exact and $\Mat_2$-stable and if $\iota$ denotes the inclusion $a\mapsto \delta_{(1,1)} a$ of $A$ into $\mK^*A$ and if $\varphi_0, \varphi_1 \colon A \to B$ are homotopic morphisms of bornological algebras, then 
$$
\mF(\iota \circ \varphi_0) = \mF(\iota \circ \varphi_1) \colon A \to \mK^*B.
$$
The same is true if we restrict ourselves to the category of Banach algebras and functors $\mF$ thereon because, in Lemma~3.26 of \cite{CMR:07}, if we start with Banach algebras we do not leave the realm of Banach algebras in the course of the proof. In particular, we get the same conclusion if $\mF$ is Morita invariant and split-exact on $\BanAlg$.

Because we can think of $\mK^*A$ as a linking algebra (up to possibly non-isometric isomorphism) of a Banach algebra Morita equivalent to $A$ and $\iota$ as the inclusion of $A$ into this linking algebra, it follows from Morita invariance of $\mF$ that $\mF(\iota)$ is an isomorphism and hence $\mF(\varphi_0) = \mF(\varphi_1)\colon \mF(A) \to \mF(B)$. So $\mF$ is homotopy invariant.
\end{proof}

\section{Relation to $\KK$-theory}\label{SectionRelationOfFToKKTheory}

In this section, let $\mF\colon \BanAlg \to \mC$ be a functor on the category of Banach algebras to some additive category that is homotopy invariant, Morita invariant and split-exact. Note that, by the universal property of Kasparov's $\KK$-theory, its restriction to the category of separable C$^*$-algebras and $*$-homomorphism factors through $\KK$-theory. We will now show that this functor factors through a natural map $\overline{\mF}$ that we define on $\KKban$.

\subsection{$\KKban$-elements give morphisms in $\mC$}


Let $A$ and $B$ be Banach algebras and let $(E,T)$ be an element of $\Eban(A,B)$, i.e., an even $\KKban$-cycle (see the proof of Theorem \ref{Theorem:KTheoryMoritaInvariant}). Fix a $k\in \N$ such that $a(T^2 -1)$ and $[a,T]$ are contained in $\Komp_B(E)$ for all $a\in A^k$. 

If $T^2 = 1$, then write $E$ as $E_0 \oplus E_1$, by degree. Let $\alpha\colon A \to \Lin_B(E_0) \subseteq \Lin_B(E)$ denote the action of $A$ on $E_0$ and $\bar{\alpha}\colon A \to \Lin_B(E_1) \subseteq \Lin_B(E)$ denote the action of $A$ on $E_1$. Since $T^2 =1$ we have another continuous homomorphism:
$$
\Ad_T \circ \bar{\alpha}\colon A \to \Lin_B(E_0) \subseteq \Lin_B(E), \quad a\mapsto T \bar{\alpha}(a) T.
$$
The condition $[a, T] \in \Komp_B(E)$ for all $a\in A^k$ 
 yields $\Ad_T \circ \bar{\alpha}(a) - \alpha(a) \in \Komp_B(E_0) \subseteq \Komp_B(E)$ for all $a\in A^k$. Hence we get a quasi-homomorphism
$$
(\alpha, \Ad_T \circ \bar{\alpha})\colon A^k \rightrightarrows \Lin_B(E) \vartriangleright \Komp_B(E).
$$
This defines a morphism $\mF(\alpha, \Ad_T \circ \bar{\alpha})$ in $\mC$ from $\mF(A^k)$ to $\mF(\Komp_B(E))$ by split-exactness of $\mF$. Now $\mF(A^k)$ is canonically isomorphic to $\mF(A)$ and $E$ is a Morita cycle from $\Komp_B(E)$ to $B$, we obtain a morphism
$$
\xymatrix{
\overline{\mF}(E,T)\colon &\mF(A) & \ar[l]_-{\cong} \mF(A^k) \ar[rr]^-{\mF(\alpha, \Ad_T \circ \bar{\alpha})} && \mF(\Komp_B(E)) \ar[r]^-{\overline{\mF}(E)} & \mF(B)
}
$$
from $\mF(A)$ to $\mF(B)$.

If $T^2 \neq 1$, then we can use the following trick which I found in the proof of Lemme 1.2.10 of \cite{Lafforgue:02}: If $E = E_0 \oplus E_1$ and if 
$$
T = \left(\begin{matrix} 0 & f\\ g & 0\end{matrix}\right),
$$
 then you can replace $E_0$ with $E_0':= E_0 \oplus E_1$ and $E_1$ with $E_1':= E_1 \oplus E_0$, where the summand $E_1$ of $E'_0$ is equal to $E_1$ as a Banach $B$-pair, but equipped with the zero-action of $A$, and the summand $E_0$ of $E_1'$ is defined analogously. Define 
$$
f' = \left(\begin{matrix} f & -(1-fg)\\ 1-gf & 2g-gfg \end{matrix}\right), \qquad g' = \left(\begin{matrix} 2g-gfg & 1-gf\\ -(1-fg) & f \end{matrix}\right).
$$
Then
$$
T' = \left(\begin{matrix} 0 & f'\\ g' & 0\end{matrix}\right)
$$
is an operator on $E':=E_0' \oplus E_1'$ such that $(T')^2 = 1$. Moreover, $(E', T') \in \Eban(A,B)$.

Define 
$$
f'' = \left(\begin{matrix} f & 0\\ 0&0 \end{matrix}\right), \qquad g'' = \left(\begin{matrix} g & 0\\ 0&0 \end{matrix}\right), 
$$
and
$$
T'' = \left(\begin{matrix} 0 & f''\\ g'' & 0\end{matrix}\right)
$$
on $E'$. Then $T'$ and $T''$ differ only by a ``compact'' operator, so they are homotopic. In particular, $(E', T')$ and $(E', T'')$ give the same element of $\KKban(A,B)$. Now $(E', T'') = (E,T) \oplus \text{trivial cycle}$, so $(E,T)$ and $(E',T')$ also give the same element of $\KKban(A,B)$. This construction can easily be checked to be very natural:

\begin{lemma}
The construction of $(E',T')$  from $(E,T)$ is compatible with the push-forward and with the sum of cycles. In particular, it respects homotopies.
\end{lemma}

\noindent So without loss of generality we can assume that $T^2 =1$.

\begin{remark}\label{Remark:kkbanTrafoOfMoritaCycles}
Consider the special case that $(E,T)$ is given by a Morita cycle $(E,\varphi)$, i.e., $T=0$, $E_0=E$, $E_1=0$, and $\alpha=\varphi$. Then it is not hard to check that $\overline{\mF}(E,T) = \overline{\mF}(E, \varphi)$ using the fact that, in this case, we can apply  Proposition \ref{Proposition:EigenschaftenDoppelspalt}, \ref{Proposition:EigenschaftenDoppelspalt:EchteHoms}., because $\bar{\alpha}=0$. 

In particular, if we consider the even more special case that $(E,T)$ is induced from a continuous homomorphism $\varphi \colon A \to B$, then it follows from Lemma~\ref{Lemma:OverlineFLiftsF} that $\overline{\mF}(E,T) = \mF(\varphi)$.
\end{remark}

\begin{lemma}\label{Lemma:NaturalTrafoNatural}
Let $(E,T) \in \Eban(A,B)$ and let $\psi\colon B\to B'$ be a homomorphism. Then
$$
\overline{\mF}(\psi_*(E,T)) = \mF(\psi) \circ \overline{\mF}(E,T) \ \colon\  \mF(A) \to \mF(B').
$$
\end{lemma}
\begin{proof}
Assume that $T^2=1$. Now consider the diagram
$$
\xymatrix{
A^k \ar@<3pt>[rr]^-{\alpha} \ar@<-3pt>[rr]_-{\Ad_T \circ \bar{\alpha}} && \Lin_B(E) \ar@{}[r]|{\vartriangleright} \ar[d]^{S \mapsto S\otimes 1}& \Komp_B(E) \ar[d]^{S \mapsto S\otimes 1}\\
 && \Lin_{B'}(\psi_*(E)) \ar@{}[r]|{\vartriangleright}& \Komp_{B'}(\psi_*(E))
}
$$
where $k\in \N$ is such that $[a, T] \in \Komp_B(E)$ for all $a\in A^k$. The upper row is the quasi-homomorphism which essentially induces $\overline{\mF}(E,T)$. Now Proposition~\ref{Proposition:EigenschaftenDoppelspalt}, 1., in conjunction with Lemma \ref{Lemma:OverlineFLiftsF} and Lemma~\ref{Lemma:MoritabanHomotopies}, 2., implies the claim.
 \end{proof}

The following lemma can be proved similarly ({\it cf.}~Lemma~\ref{Lemma:HomotopyMoritabanAndF}).

\begin{lemma}
If $(E_0,T_0)$ and $(E_1, T_1)$ are homotopic elements of $\Eban(A,B)$, then $\overline{\mF}(E_0,T_0) = \overline{\mF}(E_1, T_1)$ as morphisms from $\mF(A)$ to $\mF(B)$. In particular, we get a well-defined map $\overline{\mF}$ from $\KKban(A,B)$ to $\mC(\mF(A),\mF(B))$.
\end{lemma}

\begin{theorem}
The so-defined map $\overline{\mF}$ from $\KKban(A,B)$ to $\mC(\mF(A),\mF(B))$ is additive and natural in $A$ and $B$.
\end{theorem}
\begin{proof}
We already know from Lemma \ref{Lemma:NaturalTrafoNatural} that the map is natural in the second variable. That it is natural in the first variable is straighforward to prove. Additivity is checked as follows. 

Let $(E_1,T_1)$ and $(E_2,T_2)$ be elements of $\Eban(A,B)$ and write $(E,T)$ for $(E_1,T_1) + (E_2,T_2)$. Consider the diagram, for $i=1,2$,
$$
\xymatrix{
A^k \ar@<3pt>[rr]^-{\alpha_i} \ar@<-3pt>[rr]_-{\Ad_{T_i} \circ \bar{\alpha_i}} && \Lin_B(E_i) \ar@{}[r]|{\vartriangleright} 
\ar@{^{(}->}[d]_{\iota_i}& \Komp_B(E_i) \ar@{^{(}->}[d]_{\iota_i}\\
 && \Lin_{B}(E_1\oplus E_2) \ar@{}[r]|{\vartriangleright}& \Komp_{B}(E_1\oplus E_2)
}
$$
where the top-row defines $\overline{\mF}(E_i,T_i)$ (and $k\in \N$ is large enough). Appealing to Proposition \ref{Proposition:EigenschaftenDoppelspalt} gives that $\mF(\iota_i) \circ \overline{\mF}(\alpha_i, \Ad_{T_i} \circ \bar{\alpha}_i)$ equals the morphism induced from the the quasi-homomorphism
$$
\xymatrix{
A^k \ar@<3pt>[rr]^-{\iota_i\circ \alpha_i} \ar@<-3pt>[rr]_-{\iota_i\circ \Ad_{T_i} \circ \bar{\alpha}_i} &&
\Lin_{B}(E_1\oplus E_2) \ar@{}[r]|{\vartriangleright}& \Komp_{B}(E_1\oplus E_2).
}
$$
These two quasi-homomorphisms, for $i=1$ and $i=2$, are orthogonal and their sum is
$$
\xymatrix{
A^k \ar@<3pt>[rr]^-{\alpha} \ar@<-3pt>[rr]_-{\Ad_{T} \circ \bar{\alpha}} &&
\Lin_{B}(E_1\oplus E_2) \ar@{}[r]|{\vartriangleright}& \Komp_{B}(E_1\oplus E_2).
}
$$
where $T=T_1 \oplus T_2$ and $\alpha = \alpha_1 \oplus \alpha_2$; this is the quasi-homomorphism induced from $(E,T)$. So from Proposition \ref{Proposition:EigenschaftenDoppelspalt} we can deduce:
$$
\mF(\iota_1) \circ \overline{\mF}(\alpha_1, \Ad_{T_1} \circ \bar{\alpha}_1) + \mF(\iota_2) \circ \overline{\mF}(\alpha_2, \Ad_{T_2} \circ \bar{\alpha}_2) = \overline{\mF}(\alpha, \Ad_{T} \circ \bar{\alpha}).
$$
Now, for $i=1,2$,
$$
E \circ \Moritaban(\iota_i) = E_i, 
$$
in the Morita category, so
$$
\overline{\mF}(E) \circ \mF(\iota_i) = \overline{\mF}(E_i). 
$$
If we sum up (and neglect the isomorphism between $A^k$ and $A$), this gives:
\begin{eqnarray*}
&&\overline{\mF}(E_1, T_1) + \overline{\mF}(E_2, T_2)\\ 
&=& \ \overline{\mF}(E_1) \circ \overline{\mF}(\alpha_1, \Ad_{T_1} \circ \bar{\alpha}_1)  + \overline{\mF}(E_2) \circ \overline{\mF}(\alpha_2, \Ad_{T_2} \circ \bar{\alpha}_2)\\ 
&=& \ \overline{\mF}(E) \circ  \mF(\iota_1) \circ \overline{\mF}(\alpha_1, \Ad_{T_1} \circ \bar{\alpha}_1) + \overline{\mF}(E) \circ \mF(\iota_2) \circ \overline{\mF}(\alpha_2, \Ad_{T_2} \circ \bar{\alpha}_2)\\ 
&=& \ \overline{\mF}(E) \circ \overline{\mF}(\alpha, \Ad_{T} \circ \bar{\alpha}) = \overline{\mF}(E,T).
\end{eqnarray*}

\end{proof}

\begin{proposition}
The map  $\overline{\mF}\colon \KKban(A,B)\ \to \ \mC(\mF(A),\mF(B))$ respects the action of Morita morphisms in the second component, {\it cf.}~Section~5.7 of \cite{Paravicini:07:Morita:richtigerschienen}.
\end{proposition} 
\begin{proof}
We just consider the case of Morita equivalences, the general case being similar: Let $A$, $B$ and $B'$ be Banach algebras, let $(E,T) \in \Eban(A,B)$ and $F$ be a Morita equivalence from $B$ to $B'$. Let $L$ be the corresponding linking algebra. Let $\iota_B$ and $\iota_{B'}$ be the inclusions of $B$ and $B'$ in $L$, respectively. Then
$$
\iota_{B,*}[(E,T)] = \iota_{B',*}[(E,T) \otimes_B F] \quad \in \ \KKban(A,L).
$$
It follows that
$$
\mF(\iota_B) \circ \overline{\mF}(E,T) = \mF(\iota_{B'}) \circ \overline{\mF}((E,T) \otimes_B F) \ \colon \ \mF(A) \to \mF(L).
$$
Since $\mF(\iota_{B'})$ is invertible, we obtain
$$
\mF(\iota_{B'})^{-1}\circ \mF(\iota_B) \circ \overline{\mF}(E,T) =  \overline{\mF}((E,T) \otimes_B F)\ \colon \ \mF(A) \to \mF(B').
$$
Since $\overline{\mF}(F) = \mF(\iota_{B'})^{-1}\circ \mF(\iota_B)$, by definition, we obtain the desired equality.
\end{proof}

\begin{remark}
There is a similar natural transformation of bifunctors from $(A,B) \to \KK^{\ban}_1(A,B)$ to $(A,B) \to \mC(\mF(\Sigma A), \mF(B))$ if $\mF$ is a homotopy invariant, Morita invariant functor on $\BanAlg$ that is not only split-exact but half-exact for semi-split extensions of Banach algebras.
\end{remark}

\subsection{Relation to Kasparov's $\KK$-theory}

Recall the $\KK$-category with separable C$^*$-algebras as objects is universal for split-exact, homotopy invariant and Morita invariant functors (where these concepts have to be interpreted in the C$^*$-sense). Hence also $\mF$, restricted to the separable C$^*$-algebras, factors through $\KK$. In the proof of the universal property, a construction very similar to our construction fo $\KKban$-cycles is employed, so it is evident that the functor on $\KK$ given by the universal property factors through $\KKban$. 

However, we will give a proof of this fact here that does not compare the constructions directly. To this end, we analyse the composition of the canonical map from $\KK$ to $\KKban$ and the transformation $\overline{\mF}$ from $\KKban$ to $\mC$ constructed above. Note that this is not, at least a priori, a functor, so we cannot use the uniqueness part of the universal property right away. But it is a natural transformation of bifunctors and we will show that, for abstract reasons, it is functorial.

If $A$ and $B$ are C$^*$-algebras and if $(E,T)$ is a Kasparov cycle in $\E(A,B)$, then we can regard it as a cycle in $\Eban(A,B)$, see Section 1.6 and Proposition 1.1.4 of \cite{Lafforgue:02}. Now this cycle induces an element $\overline{\mF}(E,T)\colon \mF(A) \to \mF(B)$. As a composition of natural transformations, this transformation of bifunctors from $\KK(A,B)$ to $\mC(\mF(A), \mF(B))$ is natural. We call it $\overline{\mF}$, too.

We use the following lemma which is a somewhat abstract variant of Proposition 1.6.10 of \cite{Lafforgue:02}. 

\begin{lemma}\label{Lemma:LafforgueSkandalisFactorisationTrick}
Let $\mG$ be a functor from the category of separable C$^*$-algebras into some additive category $\mA$, and let $\tilde{\mG}$ be a natural transformation from the bifunctor $(A,B) \mapsto \KK(A,B)$ to the bifunctor $(A,B) \mapsto \Hom_{\mA}(\mG(A), \mG(B))$ (both bifunctors are considered to have values in the category of abelian groups) that extends $\mG$ in the sense that $\tilde{\mG} ([\varphi]) = \mG(\varphi)$ for all $*$-homomorphisms $\varphi$. Then $\tilde{\mG}$ is itself a functor, i.e., we have $\tilde{\mG}(y) \circ \tilde{\mG}(x) = \tilde{\mG}(y \circ x)$ for all $x\in \KK(A,B)$ and $y \in \KK(B,C)$. 
\end{lemma}
\begin{proof}
Proceed as in the proof of Proposition 1.6.10 of \cite{Lafforgue:02}. The main idea is that every element of $\KK(A,B)$ can be written as a composition of the class of a $*$-homomorphism and the inverse of a class of a $*$-homomorphism, see Lemme 1.6.11 of \cite{Lafforgue:02}. 

Note that naturality of $\tilde{\mG}$ means that, for all $x\in \KK(A,B)$ and $*$-homomorphisms $\varphi \colon A' \to A$ and $\psi \colon B\to B'$, we have
$$
\tilde{\mG} (\varphi^*\psi_* x) = \mG(\psi) \circ \tilde{\mG}(x) \circ \mG(\varphi). 
$$
 \end{proof}

\begin{lemma}\label{Lemma:mFRespctsStarHomomorphisms}
Let $A$ and $B$ be C$^*$-algebras. If $\varphi \colon A \to B$ is a $*$-homomorphism, then $\overline{\mF}$ maps the $\KK$-element $[\varphi] \in \KK(A,B)$ to $\mF(\varphi)$. 
\end{lemma}
\begin{proof}
This follows from the last paragraph of Remark \ref{Remark:kkbanTrafoOfMoritaCycles}.
\end{proof}

\begin{theorem}\label{Theorem:NaturalTrafoMultiplicative}
On the full subcategory of separable C$^*$-algebras, the natural transformation $\overline{\mF}$ of bifunctors is multiplicative, i.e., if $A$, $B$ and $C$ are separable C$^*$-algebras and if $x\in \KK(A,B)$ and $y \in \KK(B,C)$, then 
$$
\overline{\mF}(y \circ x) = \overline{\mF}(y) \circ \overline{\mF}(x).
$$
So $\overline{\mF} \colon \KK \to \mC$ is a functor that extends $\mF$ and hence identical to the functor from $\KK$ to $\mC$ given by the universal property of $\KK$. 
\end{theorem}
\begin{proof}
Lemma \ref{Lemma:mFRespctsStarHomomorphisms} implies that $\varphi \mapsto \overline{\mF}([\varphi])$ agrees with $\varphi \mapsto \overline{\mF}(\varphi)$ on all $*$-homomorphism $\varphi$. So the natural transformation of bifunctors $\overline{\mF}$ from $(A,B) \mapsto \KK(A,B)$ to $(A,B) \mapsto \mC(\mF(A), \mF(B))$ satisfies the hypotheses of Lemma \ref{Lemma:LafforgueSkandalisFactorisationTrick}, so we are done.
\end{proof}

\section{The Spanier-Whitehead construction for Banach algebras}

Let $\Cpt$ and $\Cpt_2$ denote categories of compact Hausdorff spaces and pairs of such spaces, respectively, as described in \cite{DellAmbrogio:08}. There is a natural notion of homotopy for these categories. 

\subsection{Short exact sequences}

%

Let $\mE$ be one of the classes $\mE_{\max}$ and $\mE_{\min}$. We call the morphisms $i$ appearing in the elements of $\mE$ \text{admissible monomorphisms} and the morphisms $p$ \text{admissible epimorphisms}. The class $\mE$ has the following properties: 
\begin{enumerate}
\item For every Banach algebra $A$, the sequence
\[
\xymatrix{
0\ar@{>->}[r]^-{0} & A \ar@{->>}[r]^-{\id}  & A
}
\]
is in $\mE$.

\item  For every Banach algebra $A$, the sequence
\[
\xymatrix{
A\ar@{>->}[r]^-{\id} & A \ar@{->>}[r]^-{0}  & 0
}
\]
is in $\mE$.

\item For all short exact sequences 
\[
\xymatrix{
A\ar@{>->}[r] & B \ar@{->>}[r]  & C
}
\]
in $\mE$ and every homomorphism of Banach algebras $\varphi \colon C' \to C$ the (canonical) sequence
\[
\xymatrix{
A\ar@{>->}[r] & B \times_{C} C' \ar@{->>}[r]  & C'
}
\]
is in $\mE$; in other words: $\mE$ is closed under pullbacks.
\end{enumerate}

\begin{lemma}[Three lemma]
Given a commuting diagram
\[
\xymatrix{
A\ar@{>->}[r] \ar[d]_{\varphi} & B \ar@{->>}[r]\ar[d]_{\chi}  & C\ar[d]_{\psi}\\
A'\ar@{>->}[r] & B' \ar@{->>}[r]  & C'
}
\]
with rows in $\mE$, if $\varphi$ and $\psi$ are isomorphisms, so is $\chi$.
\end{lemma}
\begin{proof}
Let the extensions be in $\mE_{\max}$. Then $\chi$ is bijective by the five-lemma and continuous and linear by assumption. So it is a homeomorphism, i.e., an isomorphism of Banach algebras.
\end{proof}

Notice that this lemma would be false if we had considered the category of Banach algebras with norm-contractive homomorphisms instead of continuous homomorphisms.

The following lemma can be proved in much the same way as Lemma A.2.3 in \cite{DellAmbrogio:08}.

\begin{lemma}\label{Lemma:IvosLemma}
Given a commuting diagram
\[
\xymatrix{
A\ar@{>->}[r] \ar[d]_{\varphi} & B \ar@{->>}[r]\ar[d]_{\chi}  & C\ar[d]_{\psi}\\
A'\ar@{>->}[r] & B' \ar@{->>}[r]  & C'
}
\]
with rows in $\mE$, then the right-hand square is a pullback if and only if $\varphi$ is an isomorphism. 
\end{lemma}

\begin{definition}
For all $(X,Y) \in\in \Cpt_2$ and $E\in \in \BanSp$ define 
$$
\Cont(X,Y;E):= \left\{ f\colon X \to E \ | \ f\restr_Y\equiv 0,\ f \text{ continuous}\right\}.
$$
This is a Banach space when equipped with the $\sup$-norm, and it is a Banach algebra in a canonical way if $E$ is a Banach algebra.
\end{definition}

Note that $\Cont(\cdot; \cdot)$ is a bifunctor from $\Cpt_2 \times \BanAlg$ to $\BanAlg$, contravariant in the first variable and covariant in the second.

\begin{proposition}
The bifunctor $\Cont(\cdot; \cdot)$ satisfies:
\begin{enumerate}
\item The category $\BanAlg$ is a module over the symmetric monoidal category $\Cpt_2$, i.e., we have canonical (isometric) isomorphisms
$$
\Cont(X,Y; \Cont (X',Y'; A)) \cong \Cont((X,Y) \wedge (X',Y');A) \text{ and } \Cont(\pt; A) \cong A
$$
for all Banach algebras $A$, natural in $(X,Y), (X',Y') \in\in \Cpt_2$.

\item For every pair $(X,Y) \in\in \Cpt_2$ the functor $\Cont(X,Y;\cdot)$ is exact in the sense that it maps elements of $\mE$ to elements of $\mE$.

\item For every pair $(X,Y) \in\in \Cpt_2$ the functor $\Cont(X,Y;\cdot)$ preserves finite products (in particular, it sends $0$ to $0$). 

\item For every Banach algebra $A$ and every triple $Z \subseteq Y \subseteq X$ of compact Hausdorff spaces, the sequence of inclusions
$$
(Y,Z) \hookrightarrow (X,Z) \hookrightarrow (X,Y)
$$
induces an extension (in $\mE$)
\[
\xymatrix{
\Cont(X,Y;A) \ar@{>->}[r] & \Cont(X,Z;A) \ar@{->>}[r]  & \Cont(Y,Z;A).
}
\]

\item For any pair $(X,Y) \in \in \Cpt_2$ and every Banach algebra $A$, the canonical projection from $(X,Y)$ to $(X/Y, \{Y\})$ induces an isomorphism
$$
\Cont(X,Y; A) \ \cong \ \Cont(X/Y, \{Y\}; A).
$$
\end{enumerate}
\end{proposition}
\begin{proof}
We sketch some parts of the proof:

As far as 2.\ is concerned in the case $\mE = \mE_{\max}$, there are certainly many different proofs, but one proceeds as follows: If $p\colon A \to B$ is a continuous surjective homomorphism, then by Michael's selection principle you can find a continuous split $s$ of $p$ such that $s(0) =0$. If $g\in \Cont(X,Y; B)$, then $s\circ g \in \Cont(X,Y;A)$ such that $p\circ (s \circ g) = (p \circ s) \circ g = g$. So the map from $\Cont(X,Y;A)$ to $\Cont(X,Y;B)$ induced by $p$ is surjective.

As far as 4.\ is concerned in the case $\mE = \mE_{\min}$: There is a continuous linear (even a completely positive) split of the canonical  $*$-homomorphism from $\Cont(X,Z;\C)$ to $\Cont(Y,Z;\C)$. This can be lifted to a continuous linear split for the canonical linear map from $\Cont(X,Z;E)$ to $\Cont(Y,Z;E)$ for any Banach space $E$, for example using the injective tensor product.

\end{proof} 

\begin{lemma}[Mayer--Vietoris, {\it cf.}~Prop.~A.2.7 in  \cite{DellAmbrogio:08}]
Consider a square in $\Cpt_2$:
\[
\xymatrix{
(X', Z') & (Y',Z') \ar@{_{(}->}[l]\\
(X,Z) \ar[u]^f & (Y,Z) \ar@{_{(}->}[l] \ar[u]_g
}
\]
where the horizontal maps are inclusions of pairs. Assume that on the bigger spaces, it gives a pushout:
\[
\xymatrix{
X' & Y'\ar@{_{(}->}[l]\\
X \ar[u]^f & Y \ar@{_{(}->}[l] \ar[u]_{g}
}
\]
Then, for every Banach algebra $A$, the following square is a pullback:
\[
\xymatrix{
\Cont(X', Z'; A) \ar@{->>}[r] \ar[d]_{\Cont(f;A)} & \Cont(Y',Z';A) \ar[d]^{\Cont(g;A)}\\
\Cont(X,Z;A)\ar@{->>}[r]   & \Cont(Y,Z;A) 
}
\]
\end{lemma}

\begin{lemma}[{\it cf.}~Lemma~A.2.9 in \cite{DellAmbrogio:08}]\label{Lemma:FunctionsAndPullback}
For every pair $(X,Y) \in \in \Cpt_2$, the functor $\Cont(X,Y;\cdot)$ preserves those pullbacks where one of the maps $A \rightarrow D \leftarrow B$ is an admissible epimorphism.
\end{lemma}

\subsection{Homotopy}

%

\begin{definition}
We define the homotopy category $\BanAlg/\!\!\sim$ to be the quotient of $\BanAlg$ by the homotopy relation, i.e., objects of $\BanAlg/\!\!\sim$ are Banach algebras and morphisms of $\BanAlg/\!\!\sim$ are homotopy classes of homomorphisms of Banach algebras. The set of morphism in this category between Banach algebras $A$ and $B$ is denoted by $[A,B]$.
\end{definition}

\begin{lemma}[ {\it cf.}~Prop.~A.3.3 of \cite{DellAmbrogio:08}]
\begin{enumerate}
\item The canonical functor $\can \colon \BanAlg \to \BanAlg/\!\!\sim$ sends finite products to finite products and $0$ to a zero object.
\item The bifunctor $\Cont(\cdot; \cdot)$ from $\Cpt_2^{\op} \times \BanAlg$ to $\BanAlg$ descends to the homotopy categories
$$
\xymatrix{
\Cpt_2^{\op} \times \BanAlg \ar[rr]^-{\Cont} \ar[d]_{\can \times \can}&& \BanAlg \ar[d]^{\can}\\
\Cpt_2/\!\!\sim^{\op} \times \BanAlg/\!\!\sim \ar[rr]^-{\Cont}&& \BanAlg/\!\!\sim
}
$$
Moreover, $\Cont(X,Y; \cdot)\colon \BanAlg/\!\!\sim \to \BanAlg/\!\!\sim$ still preserves the zero objects and finite products, for any fixed pair $(X,Y)$. In particular, this holds for the functor $\Sigma = \Cont(S^1, 1; \cdot)$:
$$
\xymatrix{
 \BanAlg \ar[rr]^-{\Sigma} \ar[d]_{\can}&& \BanAlg \ar[d]^{\can}\\
\BanAlg/\!\!\sim \ar[rr]^-{\Sigma}&& \BanAlg/\!\!\sim
}
$$
\end{enumerate}
\end{lemma}

\subsection{Inverting the suspension}

\begin{definition}
Let $A, B$ be Banach algebras and $m,n \in \Z$. Define
$$
\SW((A,m), (B,n)):= \SWban((A,m), (B,n)):=\colim_{k\to \infty} [\Sigma^{m+k} A, \ \Sigma^{n+k} B]
$$
where the connecting map in the colimit is given by suspension. With this set as morphisms, the class of all pairs $(A,m)$ with $A\in \in \BanAlg$ and $m\in \Z$ becomes a category $\SW=\SWban$. There is a canonical embedding $\can$ of $\BanAlg/\!\!\sim$ into $\SW$ given by $A \mapsto (A,0)$. Define 
$$
\Sigma\colon \SW\to \SW,\ (A,m) \mapsto (A, m+1).
$$
This is an automorphism of the category $\SW$. The notation is justified, because there is a natural isomorphism from the functor $A \mapsto (\Sigma A, 0)$ to the functor $A \mapsto (A, 1)$ implemented by $\id_A$, considered as an isomorphism from $(\Sigma A, 0)$ to $(A, 1)$; so $(A,m) \mapsto (A, m+1)$ extends $\Sigma$ from $\BanAlg/\!\!\sim$ to $\SW$.
\end{definition}

\begin{lemma}
The category $\SW$ has a zero object and all finite products (and $\can \colon \BanAlg \to \SW$ as well as $\Sigma$ preserve zero objects and finite products). Moreover, it is an additive category and the functor $\Sigma$ is additive on it. 
\end{lemma}
\begin{proof}
We've seen above that $\BanAlg/\!\!\sim$ has $0$ as a zero-object and the product of Banach algebras as a product. The suspension functor $\Sigma$ on $\BanAlg/\!\!\sim$ sends $0$ to $0$ and is compatible with products (up to isomorphism), so it follows from Lemma A.4.5 in \cite{DellAmbrogio:08} that $\SW$ also has a zero object and finite products and that both $\can\colon \BanAlg/\!\!\sim \to \SW$ and $\Sigma \colon \SW \to \SW$ preserve them. 

The sum on $[A', \Sigma B']$, for any Banach algebras $A'$ and $B'$, is given by the concatenation $\bullet$ of paths, the inverse is given by the reversal of paths, i.e., the group structure on $[A',\Sigma B']$ is induced from the canonical co-group structure on the pointed space $(S^1, 1)$.

A classical argument shows that this structure induces the structure of an abelian group on the set $[\Sigma^{m+k} A, \ \Sigma^{n+k} B]$ for all Banach algebras $A$ and $B$, for all $m,n\in \Z$ and for $k\in \N$ large enough, and that $\Sigma \colon \SW \to \SW$ preserves this group structure, {\it cf.}~Proposition A.4.7 in \cite{DellAmbrogio:08} or Lemma 6.4 in \cite{CMR:07}.

We have now shown that $\SW$ is preadditive and has a zero object, so finite products automatically provide finite sums and $\Sigma \colon \SW \to \SW$ is an additive functor.
\end{proof}

We can identify the sum of homomorphisms of Banach algebras in the special case that they are orthogonal. Recall that two homomorphisms $\varphi_1, \varphi_2\colon A \to B$ are said to be \demph{orthogonal} if, for all $a_1,a_2\in A$, we have $\varphi_1(a_1)\varphi_2(a_2) = 0 = \varphi_2(a_2) \varphi_1(a_1)$. In this case, the sum $\varphi_1 + \varphi_2$ is again a homomorphism from $A$ to $B$.

\begin{lemma}\label{Lemma:OrthogonalHomomorphisms:SigmaHo}
Let $\varphi_1, \varphi_2\colon A \to B$ be continuous orthogonal homomorphisms of Banach algebras. Then $\Sigma \varphi_1$ and $\Sigma\varphi_2$ are orthogonal homomorphisms from $\Sigma A$ to $\Sigma B$ and their sum $\Sigma \varphi_1 + \Sigma\varphi_2=\Sigma(\varphi_1 + \varphi_2)$ is homomotopic to their concatenation $\Sigma \varphi_1 \bullet \Sigma\varphi_2$. In particular,
$$
[\varphi_1 + \varphi_2] = [\varphi_1] + [\varphi_2] \quad \in \quad \SW(A,B).
$$ 
\end{lemma}
\begin{proof}
Let $H_1$ be the standard homotopy connecting $\Sigma \varphi_1$ and the concatenation $\Sigma \varphi_1 \bullet 0$ of $\Sigma \varphi_1$ and $0\colon \Sigma A \to \Sigma B$. Let $H_2$ denote the standard homotopy from $\Sigma \varphi_2$ to the concatenation $0 \bullet \Sigma \varphi_2$. Then $H_1$ and $H_2$ are orthogonal and $H_1 + H_2$ is a homotopy from $\Sigma \varphi_1 \bullet \Sigma \varphi_2 = \Sigma \varphi_1 \bullet 0 + 0 \bullet \Sigma \varphi_2$ to $\Sigma \varphi_1 + \Sigma\varphi_2$.
\end{proof}

Later on, when we define $\kkban$, we are going to identify $B$ and $\Mat_2(B)$, so the following lemma gives an alternative characterisation of the sum of homomorphisms in $\kkban$. It does not give an alternative characterisation on the level of $\SW$, but we nevertheless state the lemma here in the extent that it holds true on this level.

\begin{lemma}\label{Lemma:SummeInSWban}
Let $\varphi, \psi \colon A \to B$ be bounded homomorphisms of Banach algebras. Define 
$$
\varphi \oplus \psi \colon A \to \Mat_2(B), a \mapsto \left(\begin{matrix}\varphi(a) & 0 \\ 0 & \psi(a) \end{matrix}\right). 
$$
Let $\iota$ be the inclusion of $B$ into $\Mat_2(B)$ as the upper left-hand corner. Then
$$
[\iota] \circ ([\varphi] + [\psi]) = [\varphi \oplus \psi]
$$
in $\SW(A, \Mat_2(B))$.  
\end{lemma}
\begin{proof}
Note that, by a standard rotation argument, the homomorphism $\iota \circ \psi= \psi  \oplus 0$ is homotopic to $0\oplus \psi \colon A \to \Mat_2(B)$ (this is $\psi$ followed by the inclusion of $B$ at the lower right-hand corner of $\Mat_2(B)$). Note that $\iota \circ \varphi = \varphi \oplus 0$ and $0 \oplus \psi$ are orthogonal and $\varphi \oplus \psi = \varphi \oplus 0 + 0 \oplus \psi$, so the lemma follows. 

\end{proof}

\begin{definition}
Let $\varphi \colon A \to B$ be a continuous homomorphism of Banach algebras. The \demph{cone triangle of $\varphi$} is the following diagram in $\BanAlg$ (or its image in $\SW$):
$$
\xymatrix{
\Sigma B \ar[r]^-{\iota(\varphi)}  &\cone_{\varphi} \ar[r]^-{\epsilon(\varphi)}  & A \ar[r]^-{\varphi} &B. 
}
$$
\end{definition}

\begin{definition}\label{Definition:DistinguishedTriangle}
A \demph{distinguished triangle} in $\SW$ is a diagram
$$
\xymatrix{
\Sigma X \ar[r] &X'' \ar[r] & X' \ar[r] &X. 
}
$$
which is isomorphic in $\SW$ to the image under $(-\Sigma)^n$, for some $n\in \Z$, of some cone triangle of some continuous homomorphism of Banach algebras; here, the map $-\Sigma$ assigns to a triangle 
$$
\xymatrix{
\Sigma X \ar[r]^-{u''} &X'' \ar[r]^-{u'} & X' \ar[r]^-{u} &X 
}
$$
the triangle
$$
\xymatrix{
\Sigma \Sigma X \ar[r]^-{-\Sigma u''} &\Sigma X'' \ar[r]^-{-\Sigma u'} & \Sigma X' \ar[r]^-{-\Sigma u} &\Sigma X 
}
$$

\end{definition}

\begin{theorem}
The Spanier-Whitehead category $\SWban=\SW$ together with the inverse suspension $\Sigma^{-1}$ as translation functor and with the class of distinguished triangles defined above is a triangulated category.

\end{theorem}

The class of distinguished triangles is closed under isomorphisms of triangles by definition. And clearly, every morphism in $\SW$ fits into a distinguished triangle. Moreover, for every object $(A,m)$ of $\SW$ the following triangle is distinguished:
$$
\xymatrix{
\Sigma (A,m) \ar[r] & 0 \ar[r] & (A,n) \ar[r]^-{\id_{(A,n)}} & (A,n);
}
$$
note that $CA \sim 0$. 

\begin{lemma}
Let $\varphi \colon A \to B$ be a continuous homomorphism of Banach algebras. Then $\Sigma$ sends the pullback square defining $\cone_{\varphi}$ to the pullback square which defines $\cone_{\Sigma \varphi}$. In particular, $\Sigma \cone_{\varphi}$ and $\cone_{\Sigma \varphi}$ are canonically isomorphic.
\end{lemma}
\begin{proof}
By Lemma\ \ref{Lemma:FunctionsAndPullback}, we can conclude that $\Sigma$ sends the pullback square defining $\cone_{\varphi}$ to another pullback square, which happens to be
$$
\xymatrix{
\Sigma \cone_{\varphi} \ar[r]^-{\Sigma \epsilon(\varphi)} \ar[d] & \Sigma A \ar[d]^{\Sigma \varphi} \\
\Sigma \cone  B \ar[r]^-{\Sigma \ev^B_0} & \Sigma B.
}
$$
The Banach algebra $\Sigma \cone  B$  is isomorphic to $\cone \Sigma B$. If we identify these algebras, the above pullback square becomes the pullback square defining $\cone_{\Sigma \varphi}$, namely
$$
\xymatrix{
\cone_{\Sigma \varphi} \ar[r]^-{\epsilon(\Sigma \varphi)} \ar[d] & \Sigma A \ar[d]^{\Sigma \varphi} \\
\cone \Sigma  B \ar[r]^-{\ev^{\Sigma B}_0} & \Sigma B.
}
$$
\end{proof}

\begin{lemma} Up to isomorphism, the suspension $\Sigma \colon \BanAlg \to \BanAlg$ sends cone triangles to cone triangles. The same is true for the ``negative suspension'' $-\Sigma$ introduced in Definition \ref{Definition:DistinguishedTriangle}.
\end{lemma}
\begin{proof}
Let $\varphi \colon A \to B$ be a continuous homomorphism of Banach algebras. Then the cone triangle for $\varphi$ is 
$$
\xymatrix{
\Sigma B \ar[r]^-{\iota(\varphi)}  &\cone_{\varphi} \ar[r]^-{\epsilon(\varphi)}  & A \ar[r]^-{\varphi} &B. 
}
$$
We just discuss the case of $\Sigma$ in detail. For the case $-\Sigma$ you have to use the map $\tau\colon t\mapsto 1-t$ on $]0,1[$ which has the property that $\Sigma \varphi \circ \tau = -\Sigma \varphi$ in the $\SigmaHo(A,B)$. 

Consider the diagram
$$
\xymatrix{
\Sigma \Sigma B \ar[r]^-{\Sigma \iota(\varphi)}\ar[d]_{\gamma}^{\cong}  &\Sigma \cone_{\varphi} \ar[r]^-{\Sigma \epsilon(\varphi)} \ar[d]^{\cong} & \Sigma A \ar[r]^-{\Sigma \varphi}\ar@{=}[d] &\Sigma B\ar@{=}[d]\\
\Sigma \Sigma B \ar[r]^-{\iota(\Sigma \varphi)}  &\cone_{\Sigma \varphi} \ar[r]^-{\epsilon(\Sigma \varphi)}  & \Sigma A \ar[r]^-{\Sigma \varphi} &\Sigma B 
}
$$
The upper row is the suspended cone triangle of $\varphi$, the lower row is the cone triangle of $\Sigma \varphi$. The isomorphism between $\Sigma \cone_{\varphi}$ and $\cone_{\Sigma \varphi} $ is the one constructed above, the isomorphism $\gamma$ twists  the two copies of $\Sigma$. 

It is easy to see that the diagram commutes; the lemma follows.
\end{proof}

\begin{lemma}[Rotation Axiom] A triangle $(u'', u', u) \colon \Sigma X \to X'' \to X' \to X$ in $\SW$ is distinguished if and only if the ``rotated triangle'' $(-\Sigma u, u'', u') \colon \Sigma X' \to \Sigma X \to X'' \to X'$ is. 
\end{lemma}
\begin{proof}
Let $\varphi\colon A\to B$ be a continuous homomorphism of Banach algebras and consider the following diagram
$$
\xymatrix{
\Sigma A \ar[rr]^{(\Sigma \varphi)\circ \tau}\ar@{=}[d]  \ar@{}[rrd]|{\fA}&&\Sigma B \ar[rr]^{\iota=\iota(\varphi)}\ar@{}[rrd]|{\fB} \ar[d]^{\theta} && \cone_{\varphi} \ar[rr]^{\epsilon=\epsilon(\varphi)}\ar@{=}[d] &&A \ar@{=}[d]\ar[rr]^{\varphi}&& B\\
\Sigma A \ar[rr]_{\iota(\epsilon)}  &&\cone_{\epsilon} \ar[rr]_{\epsilon(\epsilon)}  && \cone_{\varphi} \ar[rr]_{\epsilon} && A}
$$
The upper line is the cone triangle of $\varphi$, prolongated to the left. The homomorphism $\tau \colon \Sigma A \to \Sigma A$ denotes the ``inverting morphism'' induced by $t \mapsto 1-t$, so that the canonical image of $(\Sigma \varphi) \circ \tau$ is $-\Sigma \can(\varphi)$ in $\SW$. The lower line line is the cone triangle of $\epsilon:=\epsilon(\varphi)\colon \cone_{\varphi} \to A$.

The morphism $\theta=(0,\iota)$ is given by the pullback defining $\cone_{\epsilon}$, so the square $\fB$ commutes by definition. If we can show that the square $\fA$ commutes up to homotopy and that $\theta$ is a homotopy equivalence, then we have shown the ``only if'' part of the lemma. The ``if'' implication follows by applying the ``only if'' part twice and by (de-)suspending once. 

To show that $\fA$ commutes up to homotopy we have to find a homotopy between the morphisms $\theta \circ \Sigma \varphi \circ \tau$ and $\iota(\epsilon)$ which are morphisms from $\Sigma A$ to $\cone_{\varphi}$. The desired homotopy is given in the C$^*$-algebra setting in \cite{DellAmbrogio:08}, Remark A.5.12, and the formula works just as well for Banach algebras. Or you may use the more elaborate smooth version in \cite{CMR:07}, page 111.

\end{proof}

\begin{lemma}
Let $\varphi\colon A \to B$ be a homomorphism of Banach algebras and let $\epsilon = \epsilon (\varphi) \colon \cone_{\varphi} \to A$ be the canonical (admissible) epimorphism. Then the canonical morphism $\theta = (0, \iota) \colon \Sigma B \to \cone_{\epsilon}$ is a homotopy equivalence.  
\end{lemma}
\begin{proof}
The desired homotopy inverse and the needed homotopies are given in Remark A.5.14 in \cite{DellAmbrogio:08} for C$^*$-algebras, but again, they work for Banach algebras in general. Or you could again use the smooth version from page 111 of \cite{CMR:07}.
\end{proof}

\begin{lemma}[Morphism axiom] Given in $\SW$ the solid arrow diagram 
$$
\xymatrix{
\Sigma X \ar[r] \ar[d]_{\Sigma u} & X'' \ar[r] \ar@{..>}[d]& X'\ar[r]  \ar[d] \ar@{}[dr]|{\mathfrak{C}}& X \ar[d]^u\\
\Sigma Y \ar[r] & Y'' \ar[r] & Y' \ar[r]& Y
}
$$
with distinguished lines and commutative square $\mathfrak{C}$, there always exists a dotted morphism making the two squares on its sides commute.
\end{lemma}
\begin{proof}
Proceed as in Lemma A.5.15 and Remark A.5.23 of \cite{DellAmbrogio:08}; again, the C$^*$-construction works for general Banach algebras. The exposition on page 112f of \cite{CMR:07} looks more complicated on first sight, but this is only due to the fact that the functor $J$ used there is not as compatible with the mapping cone construction as the functor $\Sigma$; the (smooth) homotopies used there work just as well.
\end{proof}

\begin{lemma}[Octahedron axiom] 
Given $\varphi\colon A \to B$, $\psi \colon B \to C$ and $\chi:= \psi \circ \varphi\colon A \to C$ in $\SW$, and given distinguished triangles $[\varphi'', \varphi', \varphi]$, $[\psi'', \psi', \psi]$ and $[\chi'', \chi', \chi]$, then there exist morphisms $\alpha\colon \cone_{\chi} \to \cone_{\psi}$ and $\beta \colon \cone_{\varphi} \to \cone_{\chi}$ such that
\begin{enumerate}
\item $[\varphi'' \psi', \beta, \alpha]$ is distinguished,
\item $\alpha \chi'' = \Sigma \psi'$ and $\chi' \beta = \varphi'$,
\item $\varphi \chi' = \psi' \alpha$ and $\chi'' \Sigma \psi = \beta \varphi''$.
\end{enumerate}
\end{lemma}
\begin{proof}
Again, one can use the arguments of \cite{DellAmbrogio:08}, Lemma A.5.24, or of Section 13.2 of \cite{CMR:07}. 
\end{proof}

\begin{theorem}[Universal Property of $\SWban$] \label{Theorem:UniversalPropertyOfSpanierWhitehead}Let $F\colon \BanAlg \to \mT$ be a functor, with values in a triangulated category, equipped with a natural isomorphism $F(\Sigma A) \cong \Sigma F(A)$ where the $\Sigma$ on the right-hand side denotes the suspension functor of $\mT$ such that
\begin{enumerate}
\item the functor $F$ is homotopy invariant;
\item the functor $F$ maps mapping cone triangles to exact triangles in $\mT$.
\end{enumerate}
Then there is a unique exact functor $\overline{F}\colon \SWban \to \mT$ such that $F= \overline{F} \circ \can$.
\end{theorem}
\begin{proof}
We sketch the proof ({\it cf.}~Theorem~A.4.4 of \cite{DellAmbrogio:08}, Proposition~6.72 of \cite{CMR:07} or \cite{KelVos:87}, \S 2). 

An object $(A,m)$ of $\SWban$ is mapped under $\overline{F}$ to the object $\Sigma^mF(A)$ of $\mT$ (note that this also makes sense for $m<0$). If a continuous homomorphism $\varphi \colon \Sigma ^{m+k} A \to \Sigma^{n+k}B$ represents a morphism in $\SWban$, then $\overline{F}[\varphi]$ is defined as $F(\varphi)$. The naturality of the isomorphism $F(\Sigma \cdot) \cong \Sigma F(\cdot)$ ensures that this definition is compatible with the inductive limit definition of the morphism sets of $\SWban$.
\end{proof}

\section{Localising $\SWban$ to obtain excision properties}\label{Section:EHoban}

We now want to define an intermediate theory $\EHoban=\EHo$ between $\SWban$ and $\kkban$ which has long exact sequences, in both variables, for short exact sequences which allow for a bounded linear split (a class that we have assembled in $\mE_{\min}$). The difference between $\EHoban$ and $\kkban$ will be that we will arrange $\kkban$ to be, in addition, invariant under Morita equivalences of Banach algebras. 

\subsection{The definition of $\EHoban$}\label{Subsection:DefinitionOfEHhoban}

We define $\EHoban$ by localising $\SWban$ at a suitable class of morphisms:

Let $\xymatrix{B\ar@{>->}[r] & E \ar@{->>}[r]^{\pi}  & A}$ be an extension of Banach algebras; let $\kappa_{\pi}\colon B \to \cone_{\pi}$ be the canonical comparison morphism given by $\kappa_{\pi}\colon B \to \cone_{\pi}, \ b \mapsto (b, 0)$. Note that if you consider the cone extension for $\pi$
$$
\xymatrix{\Sigma A \ \ar@{>->}[r]& \cone_{\pi} \ar@{->>}[r] & E\\ &B \ar[u]_{\kappa_{\pi}}&}
$$
then inverting $\kappa_{\pi}$ allows you to construct a canonical morphism from $\Sigma A$ to $B$ induced from the given extension. 

Define 
$$
\mM_{\min}:=\{\kappa_{\pi}:\ \xymatrix{B \ \ar@{>->}[r]& E \ar@{->>}[r]^{\pi}& A} \text{ extension in } \mE_{\min}\}.
$$
Similarly, define $\mM_{\max}$.

\begin{definition}
Let $\EHoban:=\EHo$ denote the triangulated category $\SWban[\mM_{\min}]$.
\end{definition}

Similarly, one can define a theory $\EHo^{\ban}_{\max}:=\SWban[\mM_{\max}]$; this theory will be related to a variant of $E$-theory in much the same way as $\EHoban$ is related to $\kkban$. In what follows, we will concentrate on $\EHoban$ rather than its ``quotient'' $\EHo^{\ban}_{\max}$.

Note that the morphisms in $\EHoban$ are not just compositions of morphisms of $\SWban$ and formal inverses of morphisms in $\mM_{\min}$; in the definition of the Verdier quotient you have to formally invert much more morphisms, namely those morphisms $\varphi$ such that the cone $\cone_{\varphi}$ is an object of the thick triangulated subcategory of $\SWban$ generated by all cones of morphisms in $\mM_{\min}$. Effectively, one does not have much control over this class.

\begin{definition}\label{Definition:EHobanepsilon}
Let $\epsilon\colon \xymatrix{B \ \ar@{>->}[r]& D \ar@{->>}[r]^{\pi}& A}$ be an extension of Banach algebras in $\mE_{\min}$. Then $\EHo(\epsilon) \in \EHo(\Sigma A, B)$ is defined as the product of the canonical morphism $\Sigma A \to \cone_{\pi}$ in $\EHo(\Sigma A, \cone_{\pi})$ and the inverse of the morphism $\kappa_{\pi}\colon B \to \cone_{\pi}$ in $\EHo(B,\cone_{\pi})$.
\end{definition}

\begin{lemma}\label{Lemma:ExtensionsAreDistinguishedInEHoban}
Let $\epsilon \colon \xymatrix{B \ \ar@{>->}[r]& D \ar@{->>}[r]^{\pi}& A}$ be in $\mE_{\min}$. Then the sequence, called \demph{extension triangle},
$$
\xymatrix{\Sigma A  \ar[r]^{\EHo(\epsilon)} & B \ar[r]& D \ar[r] & A}
$$
is an exact triangle in $\EHobancat$. In particular, every element of $\mE_{\min}$ gives long exact sequences in $\EHo$ in both variables.
\end{lemma}
\begin{proof}
The sequence is clearly a triangle. It is exact because it is isomorphic (in $\EHo$) to the mapping cone triangle $\xymatrix{\Sigma A  \ar[r] & \cone_{\pi} \ar[r]& D \ar[r] & A}$.
\end{proof}

\begin{remark}\label{Remark:AllExactTrianglesExtensionTriangles}
It is easy to see that, conversely, every exact triangle in $\EHobancat$ is isomorphic to an (iterated suspension) of an extension triangle, {\it cf.}~p.~118 in \cite{CMR:07}.
\end{remark}

\noindent The following easy lemmas will be used at several instances in what follows.

\begin{lemma}\label{Lemma:ExactSequencesWithKContractiveMiddleTerm}
Let $\epsilon\colon \xymatrix{B \ \ar@{>->}[r]& D \ar@{->>}[r]& A}$ be an extension of Banach algebras in $\mE_{\min}$ such that $D \cong 0$ in $\EHo$. Then $\EHo(\epsilon)$ is an isomorphism from $\Sigma A$ to $B$. 
\end{lemma}
\begin{proof}
Use the long exact sequences in both variables.
\end{proof}

\begin{lemma}\label{Lemma:kkbanVonErweiterungenLeiterdiagramm}
Let $\epsilon\colon \xymatrix{B \ \ar@{>->}[r]& D \ar@{->>}[r]^{\pi}& A}$ and $\epsilon'\colon \xymatrix{B' \ \ar@{>->}[r]& D' \ar@{->>}[r]^{\pi'}& A'}$ be extensions in $\mE_{\min}$ which can be put in a diagram of the form
$$
\xymatrix{B \ \ar@{>->}[r] \ar[d]_{\psi}& D \ar@{->>}[r]^-{\pi}\ar[d]& A\ar[d]^{\varphi} \\ B' \ \ar@{>->}[r]& D' \ar@{->>}[r]^-{\pi'}& A'}
$$
Then 
$$
\EHo(\Sigma \varphi) \cdot \EHo(\epsilon') =  \EHo(\epsilon) \cdot \EHo(\psi) \ \in\ \EHo(\Sigma A, B'),
$$
i.e., we can extend the above commutative diagram to the left to a commutative diagram in $\EHobancat$:
$$
\xymatrix{\Sigma A \ar[r]^{\EHo(\epsilon)} \ar[d]_{\Sigma \varphi}& B \ \ar@{>->}[r] \ar[d]_{\psi}& D \ar@{->>}[r]\ar[d]& A\ar[d]^{\varphi} \\ \Sigma A' \ar[r]^{\EHo(\epsilon')} &B' \ \ar@{>->}[r]& D' \ar@{->>}[r]& A'}
$$
\end{lemma}
\begin{proof}
There is a canonical map $\chi\colon \cone_{\pi} \to \cone_{\pi'}$ which can be obtained using the universal properties of the pullback. It fits into a diagram
$$
\xymatrix{\Sigma A \ \ar[r] \ar[d]_{\Sigma \varphi}& \cone_{\pi} \ar[d]^{\chi}& B\ar[d]^{\psi} \ar_-{\cong}[l] \\ \Sigma A' \ \ar[r]& \cone_{\pi'} & B' \ar_-{\cong}[l]}
$$
The fact that this diagram commutes implies the claim. 
\end{proof}

\begin{proposition}\label{Proposition:ActionOfSpacesOnEHoban}
Let $A$ be a Banach algebra and $(X,X_0)$ an object in $\CW_2$. The natural homomorphism from $\Cont_0(X-X_0) \otimes A$ to $\Cont_0(X-X_0, A)$ is an isomorphism in $\EHo$.
\end{proposition}
\begin{proof}
You can use Proposition \ref{Proposition:ActionOfSpacesOnFunctors} on the functors $A \mapsto \EHo(D, (A, k))$ and  $A \mapsto \EHo((A, k),D)$ for $k\in \Z$ and any fixed object $D$ of $\EHobancat$.
\end{proof}

\begin{definition}
A \demph{triangulated homology functor for Banach algebras} is a functor $F$ from $\BanAlg$ to a triangulated category $\mT$ together with a natural isomorphism $F(\Sigma(A)) \cong \Sigma(F(A))$, where the $\Sigma$ on the right-hand side denotes the suspension functor of $\mT$, such that
\begin{enumerate}
\item the functor $F$ is homotopy invariant;
\item the functor $F$ maps mapping cone triangles to exact triangles in $\mT$;
\item the functor $F$ is half-exact.
\end{enumerate}
\end{definition}

\begin{remark}
Note that every triangulated homology functor is split-exact by Corollary \ref{Corollary:HalfexactnessLongExactSequence:Additive}.
\end{remark}

\begin{theorem}[Universal property of $\EHoban$, {\it cf.}~Prop.~7.72 of \cite{CMR:07}]\label{Theorem:UniversalPropertyEHoban} The category $\EHobancat$ is triangulated and the canonical functor 
$$
\EHobanfunc\colon \SWbancat \to \EHobancat
$$
is exact. Every extension in $\mE_{\min}$ gives a distinguished triangle in $\EHobancat$, and all distinguished triangles in $\EHobancat$ are isomorphic to iterated suspensions of extension triangles induced by extensions in $\mE_{\min}$. 

The functor $\EHobanfunc\colon \BanAlg \to \EHobancat$ is a triangulated homology functor for Banach algebras; it is the universal triangulated homology theory: If $F\colon \BanAlg \to \mT $ is a triangulated homology functor for Banach algebras, then there is a unique exact functor $\overline{F} \colon \EHobancat \to \mT$ such that $F = \overline{F} \circ \EHobanfunc$. And every exact functor on $\EHobancat$ gives a triangulated homology functor in this way.


Let $(F_k)_{k\in \Z}$ be a homology theory for Banach algebras with values in some abelian category $\mA$, then $\overline{F}(A,k):=F_k(A)$ defines a homological functor $\overline{F}\colon \EHobancat \to \mA$. Conversely, any such homological functor arises from a unique homology theory for Banach algebras in this fashion.

%
%
\end{theorem}
\begin{proof}
The claims of the first paragraph are either true by construction or follow from Lemma \ref{Lemma:ExtensionsAreDistinguishedInEHoban} and Remark \ref{Remark:AllExactTrianglesExtensionTriangles}. That $\EHobanfunc\colon \BanAlg \to \EHobancat$ is homotopy invariant and maps mapping cone triangle to exact triangles is also clear from the construction; by Lemma \ref{Lemma:ExtensionsAreDistinguishedInEHoban} it is a triangulated homology theory.

Composing $\EHobanfunc$ with an exact functor from $\EHobancat$ into some triangulated category hence gives a triangulated homology theory on $\BanAlg$. Conversely, let $F\colon \BanAlg \to \mT $ be a triangulated homology functor of Banach algebras. 
From Theorem \ref{Theorem:UniversalPropertyOfSpanierWhitehead} it follows that $F$ lifts uniquely to an exact functor on $\SWban$. Corollary \ref{Corollary:KappaIsomorphism:Additive} implies that $F(\kappa_{\pi}) \colon F(B) \to F(\cone_{\pi})$ is an isomorphism whenever $\xymatrix{B \ \ar@{>->}[r]& D \ar@{->>}[r]^{\pi}& A}$ is a split-exact extension of Banach algebras. The universal property of the localisation construction of triangulated categories then implies that $F$ factors uniquely through an exact functor on $\EHobancat$.

The arguments for homology theories are similar, {\it cf.}~Prop.~6.72 in \cite{CMR:07}.
\end{proof}

%
%
%
%
%

\begin{lemma}[ {\it cf.}~p.~122 of \cite{CMR:07}]\label{Lemma:SigmaProduktInEHoban}
Let $D$ be a Banach algebra. Let $\sigma_D$ be the functor that assigns to every Banach algebra $A$ the Banach algebra $A \otimes D$, where $\otimes$ denotes the completed projective tensor product, and to a homomorphism $\varphi$ the homomorphism $\varphi \otimes \id_D$. Then $\sigma_D$, as a functor from the Banach algebras to $\EHo$, is a triangulated homology functor. It hence factors through $\EHo$. The resulting functor from $\EHo$ to itself will be called $\sigma_D$, too.
\end{lemma}
\begin{proof}
Let $\varphi \colon A \to B[0,1]$ be a homotopy between homomorphism of Banach algebras $\varphi_0$ and $\varphi_1$. Then $\sigma_D(\varphi)$ is a homomorphism from $A\otimes D$ to $B[0,1] \otimes D$. There is a canonical homomorphism $\iota$ from $B[0,1]\otimes D$ to $(B \otimes D)[0,1]$ that is compatible with the evaluation maps. Hence $\iota\circ \sigma_D(\varphi)$ is a homotopy from $\sigma_D(\varphi_0)$ to $\sigma_D(\varphi_1)$. In particular, $\sigma_D$ is homotopy invariant.

It is clear that $\sigma_D$ respects semi-split extensions, so in particular, it is half-exact for semi-split extensions. 

Now consider the extension $\xymatrix{(\Sigma B) \otimes D\ \ar@{>->}[r] & (\cone B) \otimes D \ar@{->>}[r] & B \otimes D}$ for some Banach algebra $B$. Because the algebra in the middle is contractible, we obtain an isomorphism $\Sigma (B \otimes D) \cong (\Sigma B) \otimes D$ which is given by the canonical homomorphism from $(\Sigma B) \otimes D$ to $\Sigma (B \otimes D)$.
Let $\varphi \colon A \to B$ be a homomorphism of Banach algebras. Consider the following diagram
$$
\xymatrix{(\Sigma B) \otimes D\ \ar@{>->}[r] \ar[d]_{\cong}& \cone_\varphi \otimes D \ar@{->>}[r] \ar[d]& A \otimes D\ar@{=}[d] \ar[r] & B \otimes D \ar@{=}[d] \\
\Sigma (B \otimes D)\ \ar@{>->}[r] & \cone_{\varphi \otimes \id_D} \ar@{->>}[r] & A \otimes D\ar[r] & B \otimes D}
$$
This diagram commutes in $\BanAlg$. The lower row is a cone triangle, and the upper row is isomorphic to it in $\EHobancat$. So the upper row is an exact triangle in $\EHobancat$. Hence $\sigma_D$ maps mapping cone triangles to exact triangles.
\end{proof}

\subsection{Connection to  stabilisation with the functor $J$}\label{Subsection:FunctorJ}

\begin{definition}\label{Definition:BanachTensorAlgebra}
Let $X$ be a Banach space. Define
$$
T_1X:= \ell^1-\bigoplus_{n\in \N} X^{\oplus_{\pi}^n}
$$
where $X^{\oplus_{\pi}^n}$ denotes the $n$-fold projective tensor product of $X$ with itself, and $\ell^1-\bigoplus$ denotes the (completed) infinite sum in the category of Banach spaces. 

Let $r>0$ and let $X_r$ denote the Banach space $X$, but with the norm $\norm{\cdot}$ of $X$ replaced with the (equivalent) norm $r \norm{\cdot}$. Define 
$$
T_rX:= T_1 X_r.
$$
\end{definition}

The space $T_rX$ is a completion of the algebraic tensor algebra of $X$, and it is a Banach algebra with the induced multiplication. There is a canonical inclusion of $X$ into $T_r X$ of norm $r$. The Banach algebra $T_rX$ has the universal property for continuous linear maps of norm at most $r$ from $X$ into some Banach algebra. 

If $0< r_1 \leq r_2$, then there is a canonical continuous inclusion of the algebra $T_{r_2} X$ into $T_{r_1}X$, prolonging the identity on the algebraic tensor algebra and of norm $\leq r_1 / r_2$. The projective limit for $r\to \infty$ over $T_rX$ is a Fr\'{e}chet algebra and has the universal property for all continuous linear maps from $X$ into some Banach algebra (or Fr\'{e}chet algebra); {\it cf.}~\cite{Cuntz:97}.

\begin{definition}
Let $A$ be a Banach algebra and $r\geq 1$. Then the identity map on $A$ induces a canonical continuous homomorphism of Banach algebras from $T_r A$ to $A$ of norm $\leq 1/r$ (note that it factors through $T_1A$) with continuous linear split of norm $\leq r$. Let 
$$
J_r A := \kernel \left(T_r A \to A \right).
$$
We hence obtain a short exact sequence 
$$
\xymatrix{J_r A \ \ar@{>->}[r]& T_r A \ar@{->>}[r]& A}
$$
for every choice of $r\geq 1$.
\end{definition}

\noindent Because $T_r A$ is contractible, Proposition \ref{Proposition:HalfexactnessLongExactSequence} implies the following observation.

\begin{proposition}[{\it cf.}~Lemma~4.1.5 of \cite{CuntzThom:06}]\label{Proposition:HalfexactnessJAndSigma}
Let $F$ be a functor on $\BanAlg$ with values in an additive category that is half-exact for all semi-split extensions and homotopy invariant. Then, for every $r\geq 1$, there is a natural isomorphism $F(\Sigma A)\cong F(J_r A)$ for all Banach algebras $A$. Hence $F(J_{r_1}A) \cong F(J_{r_2}A)$ for all $r_1, r_2 \geq 1$, naturally. 
\end{proposition}

\begin{corollary}\label{Corollary:AllJsTheSame}
For every $r\geq 1$, there is a natural isomorphism $\Sigma A \cong J_r A$ in $\EHobancat$. Moreover, $J_{r_1}A \cong J_{r_2} A$ in $\EHobancat$ for all $r_1,r_2 \geq 1$.
\end{corollary}

\begin{proposition}
Let $\epsilon\colon \xymatrix{B \ \ar@{>->}[r]& D \ar@{->>}[r]& A}$ be a short exact sequence of Banach algebras with continuous linear split $\sigma \colon A \to D$. Let $r \geq 1$ and assume that $\norm{\sigma} \leq r$. Then, by the universal property of $T_r A$, the split $\sigma$ induces a continuous linear homomorphism $\hat{\sigma}$ from $T_rA$ to $D$ such that the following square commutes
$$
\xymatrix{
D \ar@{->>}[r] & A \\
T_r A \ar@{->>}[r] \ar[u]_{\hat{\sigma}}& A \ar@{=}[u]
}
$$
It can be completed to a commutative diagram
$$
\xymatrix{
B \ar@{>->}[r] &D \ar@{->>}[r] & A \\
J_rA \ar@{>->}[r] \ar[u]^{\gamma_r(\epsilon)} & T_r A \ar@{->>}[r] \ar[u]_{\hat{\sigma}}& A \ar@{=}[u]
}
$$
The map $\gamma_r(\epsilon) \colon J_r A \to B$ is called the classifying map of the extension $\xymatrix{B \ar@{>->}[r] &D \ar@{->>}[r] & A}$.
\end{proposition}

In other words: the extension $\xymatrix{J_r A \ \ar@{>->}[r]& T_r A \ar@{->>}[r]& A}$ is universal among the extensions of Banach algebras with continuous linear split of norm at most $r$.

\begin{remark} Let $r\geq 1$ and $A$ a Banach algebra. You can identify the natural isomorphism $\Sigma A \cong J_r A$ in $\EHobancat$ of \ref{Corollary:AllJsTheSame} as the inverse of the classifying map $\Lambda_r\colon J_r A \to \Sigma A$ of the extension $\xymatrix{\Sigma A \ar@{>->}[r] &\cone A \ar@{->>}[r] & A}$ using a five-Lemma argument. Similarly you can identify the natural isomorphism $J_{r_2} A \cong J_{r_1}A$ in $\EHobancat$ for $1 \leq r_1 \leq r_2$ as the one inherited from the natural map $T_{r_2}A \to T_{r_1}A$.
\end{remark}

\begin{remark}\label{Remark:AllJsTheSame}
You could define $\EHobancat$ also by inverting the morphisms $\Lambda_r\colon J_r A \to \Sigma A$, for $r\geq 1$ and $A$ a Banach algebra, in $\SWbancat$: By arguing as in Satz 5.3 of \cite{Cuntz:97}, see also Theorem 6.63 of \cite{CMR:07}, you can show that the resulting triangulated category is the same as the one constructed here.
\end{remark}

\begin{proposition}\label{Proposition:AlternativenGleich}
Let $\epsilon\colon \xymatrix{B \ \ar@{>->}[r]& D \ar@{->>}[r]^{\pi} & A}$ be a short exact sequence of Banach algebras with continuous linear split $\sigma \colon A \to D$. Let $r \geq 1$ and assume that $\norm{\sigma} \leq r$. Then 
$$
\EHobanfunc(\epsilon) \circ \EHobanfunc(\Lambda_r) = \EHobanfunc(\gamma_r(\epsilon)) \quad \in \quad \EHobancat(J_r A, B), 
$$
i.e., if we identify $J_r A$ and $\Sigma A$ via $\Lambda_r$, then it does not matter whether we define the element of $\EHobancat$ associated to the extension $\epsilon$ via the construction of $\EHoban(\epsilon)$ given above or via the classifying map $\gamma_r(\epsilon)$. 
\end{proposition}
\begin{proof}
To streamline the notation we skip the subscript $r$ and write $J$ instead of $J_r$ etc..

Consider the commutative diagram
$$
\xymatrix{J A \ar@{>->}[r] \ar[d]_{\gamma(\epsilon)} & T A\ar[d] \ar@{->>}[r] & A \ar@{=}[d] \\ 
B \ar@{>->}[r] & D \ar@{->>}[r] & A}
$$
or rather its suspended version
$$
\xymatrix{\Sigma J A \ar@{>->}[r] \ar[d]_{\Sigma \gamma(\epsilon)} & \Sigma T A\ar[d] \ar@{->>}[r] & \Sigma A \ar@{=}[d] \\ 
\Sigma B \ar@{>->}[r] & \Sigma D \ar@{->>}[r] & \Sigma A}
$$
From this diagram and the analogue of Lemma 3.3 of \cite{Cuntz:97} for Banach algebras, we can deduce that 
$$
\gamma(\Sigma \epsilon) = \Sigma \gamma(\epsilon) \circ \Delta_A
$$
where $\Sigma \epsilon$ denotes the suspended version of the extension $\epsilon$ and $\Delta_A$ denotes the canonical homomorphism $J \Sigma A \to \Sigma J A$. Now consider the diagram in $\BanAlg/ \sim$:
$$
\xymatrix{
&JJA \ar[dl]_{J\Lambda_A} \ar[dr]^{\Lambda_{JA}}&\\
J \Sigma A\ar [rr]^{\Delta_A} \ar[d]_{J \iota} \ar[rrd]^{\gamma(\Sigma \epsilon)}&& \Sigma JA \ar[d]^{\Sigma \gamma(\epsilon)}\\
J C_{\pi} \ar[rr]_{f} \ar[rd]_{\Lambda_{C_{\pi}}}&& \Sigma B \ar[ld]^{\Sigma \kappa_{\pi}}\\
& \Sigma C_{\pi} 
}
$$
We have already analysed the right central triangle and we have shown that it is commutative. The upper triangle is commutative by the arguments given in the proof of Lemma 6.30 of \cite{CMR:07}; the morphisms $\Lambda_A$ and $\Lambda_{JA}$ are defined as in Remark \ref{Remark:AllJsTheSame}. 

The lower left triangle central needs some explanation: The morphism $\iota \colon \Sigma A \to \cone_{\pi}$ is the canonical embedding. The morphism $f$ is defined as the classifying map of a certain extension $ \xymatrix{\Sigma B \ \ar@{>->}[r]& \cone D \ar@{->>}[r] & \cone_{\pi}}$ introduced in the proof of Theorem 6.63 of \cite{CMR:07}; it is the inverse of $\kappa_{\pi} \colon B \to \cone_{\pi}$ up to suspension with $J$ and $\Sigma$ ({\it cf.}~the proof of Theorem~6.63 in \cite{CMR:07}), in particular, the lower triangle is commutative. It follows that it suffices to show that the central left triangle is commutative. But one can easily put the extension  $ \xymatrix{\Sigma B \ \ar@{>->}[r]& \cone D \ar@{->>}[r] & \cone_{\pi}}$ into a commutative diagram
$$
\xymatrix{
\Sigma B \ \ar@{>->}[r] \ar@{=}[d]& \Sigma D \ar@{->>}[r] \ar[d]& \Sigma A \ar[d]^{\iota} \\
\Sigma B \ \ar@{>->}[r]& \cone D \ar@{->>}[r] & \cone_{\pi}
}
$$
But this implies the commutativity of the central left triangle by the Banach algebra version of Lemma 3.3 of \cite{Cuntz:97}. 

So the complete diagram is commutative. 

To prove the proposition it suffices to compare the homomorphisms $\iota \circ \Lambda_A$ and $\kappa_{\pi} \circ \gamma(\epsilon)$ from $JA$ to $C_{\pi}$. The diagram shows that $\Lambda (\iota \circ \Lambda_A)$ is homotopic to $\Lambda(\kappa_{\pi} \circ \gamma(\epsilon))$ where we use $\Lambda$ notation from Section 6.3 of \cite{CMR:07}. A short calculation shows that this implies that  $\Sigma (\iota \circ \Lambda_A)$ is equal to $\Sigma (\kappa_{\pi} \circ \gamma(\epsilon))$ in $\EHobancat$.

\end{proof}

\section{A definition of $\kkban$}

\subsection{Localisation of $\SW$}

We now want to define a theory of KK-type on the Banach algebras: semi-split extensions should give long exact sequences in both variables and Morita equivalences should give isomorphisms. We hence localise $\EHoban$ (or rather $\SWban$) accordingly:

%
%
Given a Morita equivalence ${_A}E_B$, consider the morphism
$$
\iota_E \colon B \hookrightarrow \left(\begin{matrix}A & E^>\\ E^< & B\end{matrix}\right).
$$
Define 
$$
\mM_{\text{Morita}}:=\{\iota_{E}:\ {_A}E_B \text{ Morita equivalence}\}.
$$ 


\begin{definition}
Let $\kkban$ denote the triangulated category $\SWban[\mM_{\min} \cup \mM_{\text{Morita}}]$.
\end{definition}

The class $\mM_{\min}$ was defined in Subsection \ref{Subsection:DefinitionOfEHhoban}. Note that there is a commuting diagram of canonical functors
$$
\xymatrix{
\SWban \ar[d] \ar[dr]\\
\EHoban \ar[r] & \kkban
}
$$ 


\begin{question}\label{Question:InvertFullMorphisms} Is it feasible to invert the following type of morphisms between Banach algebras: $\varphi \colon A \to B$ such that $\varphi(A)$ is dense in $B$ and such that $\varphi^{-1}(B^{-1}) = A^{-1}$? 
\end{question}

The following results, concepts and isomorphisms (\ref{Deflemma:ExtensionsAreDistinguishedInkkban} to \ref{Corollary:TwoPossibleSuspensionsInkkban}) carry over from $\EHobancat$, see Section \ref{Section:EHoban} for proofs.

\begin{deflemma}\label{Deflemma:ExtensionsAreDistinguishedInkkban}
Let $\epsilon\colon \xymatrix{B \ \ar@{>->}[r]& D \ar@{->>}[r]^{\pi}& A}$ be an extension of Banach algebras in $\mE_{\min}$. Then $\kkban(\epsilon) \in \kkban(\Sigma A, B)$ is defined as the product of the canonical morphism $\Sigma A \to \cone_{\pi}$ in $\kkban(\Sigma A, \cone_{\pi})$ and the inverse of the morphism $\kappa_{\pi}\colon B \to \cone_{\pi}$ in $\kkban(B,\cone_{\pi})$.

The extension triangle
$$
\xymatrix{\Sigma A  \ar[r]^{\kkban(\epsilon)} & B \ar[r]& D \ar[r] & A}
$$
is an exact triangle in $\kkban$. In particular, every element of $\mE_{\min}$ gives long exact sequences in $\kkban$ in both variables.
\end{deflemma}


\begin{lemma}\label{Lemma:ExactSequencesWithKContractiveMiddleTerm:kkban}
Let $\epsilon\colon \xymatrix{B \ \ar@{>->}[r]& D \ar@{->>}[r]^{\pi}& A}$ be an extension of Banach algebras in $\mE_{\min}$ such that $D \cong 0$ in $\kkban$. Then $\kkban(\epsilon)$ is an isomorphism from $\Sigma A$ to $B$. 
\end{lemma}

\begin{proposition} Let $r\geq 1$ and $A$ a Banach algebra. Let $\Lambda_r\colon J_r A \to \Sigma A$ denote the classifying map of the extension $\xymatrix{\Sigma A \ar@{>->}[r] &\cone A \ar@{->>}[r] & A}$. Then $\Lambda_r\colon J_r A \to \Sigma A$ is an isomorphism in $\kkban$. In particular, $J_{r_1}A \cong J_{r_2} A$ in $\kkban$ for all $r_1,r_2 \geq 1$.
\end{proposition}

\begin{lemma}\label{Lemma:kkbanVonErweiterungenLeiterdiagramm:kkban}
Let $\epsilon\colon \xymatrix{B \ \ar@{>->}[r]& D \ar@{->>}[r]^{\pi}& A}$ and $\epsilon'\colon \xymatrix{B' \ \ar@{>->}[r]& D' \ar@{->>}[r]^{\pi'}& A'}$ be extensions in $\mE_{\min}$ which can be put in a diagram of the form
$$
\xymatrix{B \ \ar@{>->}[r] \ar[d]_{\psi}& D \ar@{->>}[r]\ar[d]& A\ar[d]^{\varphi} \\ B' \ \ar@{>->}[r]& D' \ar@{->>}[r]& A'}
$$
Then 
$$
\kkban(\Sigma \varphi) \cdot \kkban(\epsilon') =  \kkban(\epsilon) \cdot \kkban(\psi) \ \in\ \kkban(\Sigma A, B').
$$
\end{lemma}

\begin{proposition}\label{Proposition:ActionOfSpacesOnkkban}
Let $A$ be a Banach algebra and $(X,X_0)$ an object in $\CW_2$. The natural homomorphism from $\Cont_0(X-X_0) \otimes A$ to $\Cont_0(X-X_0, A)$ is an isomorphism in $\kkban$.
\end{proposition}

\begin{corollary}\label{Corollary:TwoPossibleSuspensionsInkkban} For every Banach algebra $A$, the natural homomorphism $\Sigma \otimes_{\pi} A \to \Sigma A$ is an isomorphism in $\kkban$.
\end{corollary}

\begin{theorem}
The category $\kkban$ satisfies Bott periodicity: There is a natural isomorphism $\Sigma^2 A \cong A$ in $\kkban$ for all Banach algebras $A$.
\end{theorem}
\begin{proof}
This follows from Theorem \ref{Theorem:BottPeriodicity:Functor}.
\end{proof}

The following lemma follows directly from Lemma \ref{Lemma:OrthogonalHomomorphisms:SigmaHo}.

\begin{lemma}\label{Lemma:OrthogonalHomomorphisms:kkban}
Let $\varphi, \psi \colon A \to B$ be continuous orthogonal homomorphisms of Banach algebras. 
$$
\kkbanfunc(\varphi + \psi) = \kkbanfunc(\varphi) + \kkbanfunc(\psi) \quad \in \quad \kkban(A,B).
$$ 
\end{lemma}

As a consequence of this Lemma or Lemma \ref{Lemma:SummeInSWban} we have the following alternative expression for the sum of two homomorphisms in $\kkban$.

\begin{lemma}\label{Lemma:SummeInkkban}
Let $\varphi, \psi \colon A \to B$ be bounded homomorphisms of Banach algebras. Define 
$$
\varphi \oplus \psi \colon A \to \Mat_2(B), a \mapsto \left(\begin{matrix}\varphi(a) & 0 \\ 0 & \psi(a) \end{matrix}\right). 
$$
Let $\iota$ be the inclusion of $B$ into $\Mat_2(B)$ as the upper left-hand corner. Then
$$
[\iota] \circ ([\varphi] + [\psi]) = [\varphi \oplus \psi]
$$
in $\kkban(A, \Mat_2(B))$. In other words: If we identify $B$ and $\Mat_2(B)$ as objects of $\kkban$, then $\varphi \oplus \psi$ realises the sum of $\varphi$ and $\psi$ in $\kkban(A,B)$. 
\end{lemma}

\begin{theorem}[Universal property of $\kkban$ I] The category $\kkban$ is triangulated and the canonical functor 
$$
\kkban\colon \SWbancat \to \kkban
$$
is exact. Every extension in $\mE_{\min}$ gives a distinguished triangle in $\kkban$. The functor $\kkban\colon \BanAlg \to \kkban$ is a Morita invariant triangulated homlogy functor and split exact.


There is a canonical bijection between Morita invariant triangulated homology functors for Banach algebras and exact functors on $\kkban$.

%
%
\end{theorem}
\begin{proof}
This is proved just as or using Theorem \ref{Theorem:UniversalPropertyEHoban}.
\end{proof}

\begin{lemma}\label{Lemma:SigmaProduktInkkban}
Let $D$ be a Banach algebra. Let $\sigma_D$ be the functor from $\BanAlg$ to $\kkban$ defined by $B \mapsto B\otimes D$ as in Lemma \ref{Lemma:SigmaProduktInEHoban}. Then $\sigma_D$ is a Morita invariant triangulated homology functor and thus factors through $\kkban$. The resulting functor from $\kkban$ to itself will also be called $\sigma_D$.
\end{lemma}
\begin{proof}
%
Lemma \ref{Lemma:SigmaProduktInEHoban} implies that $\sigma_D$ is a triangulated homology functor. We check that it is Morita invariant: 

Let $L$ be the linking algebra of a Morita equivalence between two Banach algebras $A$ and $B$. Then it is easy to see that $L \otimes D$ is the linking algebra for the corresponding Morita equivalence between $A \otimes D$ and $B\otimes D$; moreover, $\sigma_D$ of the canonical injection of $A$ into $L$ is the canonical injection of $A \otimes D$ into $L\otimes D$. So $\sigma_D$ is Morita invariant.
\end{proof}

\begin{theorem}[Universal property of $\kkban$ II, {\it cf.}~Theorem~7.26 of \cite{CMR:07}] 

Let $F$ be any functor from the category $\BanAlg$ to an additive category that is homotopy invariant, Morita invariant, and half-exact for semi-split extensions. Then $F$ factors uniquely through $\kkbanfunc$.  


Let $F$ and $F'$ be functors with the above properties, so that they descend to functors $\overline{F}$ and $\overline{F}'$ on $\kkban$. If $\Phi\colon F \to F'$ is a natural transformation, then $\Phi$ remains natural with respect to morphisms in $\kkban$, that is $\Phi$ is a natural transformation $\overline{F} \to \overline{F}'$.
\end{theorem}
\begin{proof}
Note that $F$ satisfies Bott periodicity by Theorem \ref{Theorem:BottPeriodicity:Functor}. Define functors
$$
F_k(A) := F(\Sigma^k A),
$$
where $A$ is a Banach algebra and $k\in \N_0$. We can extend this definition to $k\in \Z$ by periodicity and we have long exact sequences for $(F_k)_{k\in\N}$ and semi-split extensions by Proposition \ref{Proposition:HalfexactnessLongExactSequence}. So $(F_k)_{k\in \N}$ is a homology theory in the sense of Definition \ref{Definition:HomologyTheory}.

Now the universal property of $\EHoban$, i.e., Theorem \ref{Theorem:UniversalPropertyEHoban}, implies that $(F_k)_{k\in \Z}$ factors uniquely through $\EHoban$. Now $F=F_0$ is Morita invariant, so it factors uniquely through $\kkban$, see Proposition 2.1.24 of \cite{Neeman:01}.

To see that $\Phi\colon F \to F'$ is also natural for morphisms in $\kkban$ observe that every such morphism can be represented by compositions of continuous homomorphisms of Banach algebras and inverses of such homomorphisms. More precisely, a morphism $f\in \kkban(A,B)$ can be represented by a diagram 
$$
\xymatrix{\Sigma ^k A \ar[r]^{f_1} & D & \Sigma ^k B \ar[l]_{f_2}^{\cong}}
$$
where $f_1$ and $f_2$ are homomorphisms of Banach algebras. 
Without loss of generality you can assume that $k$ is even. The isomorphism $\Sigma^k A \cong A$ in $\kkban$ is also given by the comoposition of (inverses) of homomorphisms of Banach algebras, and similarly for $B$. 
\end{proof}

\begin{corollary}
The $\KTh$-functor on $\BanAlg$ factors uniquely through $\kkban$.
\end{corollary}

\subsection{What is $\kkban(\C, B)$?}

\begin{theorem} There is a natural isomorphism
$$
\KTh_0(B) \cong \kkban(\C, B)
$$
for every Banach algebra $B$. It maps the class of an idempotent $e \in B$ to the $\kkban$-class of the homomorphism $\C \to B, \lambda \mapsto \lambda e$.
\end{theorem}
\begin{proof}
First consider the case that the algebra $B$ is unital. Then any class in $\KTh_0(B)$ comes from an idempotent in some $\Mat_n(B)$ and hence gives rise to a bounded homomorphism $\C \to \Mat_n(B)$. Homotopy stability of $\kkban$ implies that similar idempotents in $\Mat_n(B)$ give rise to the same class in $\kkban(\C, \Mat_{2n}(B))$. Because $\kkban$ is $\Mat_{*}$-stable, we obtain a well-defined class in $\kkban(\C,B)$; and Lemma \ref{Lemma:SummeInkkban} implies that we get a well-defined natural map $\alpha_B \colon \KTh_0(B) \to \kkban(\C,B)$. Using split-exactness of $\KTh_0$ and $\kkban$ for the extension  $0 \to B \to \unital{B} \to \C \to 0$ we extend this natural transformation $\alpha_B$ to non-unital algebras. 

The map $\alpha_B\colon \KTh_0(B) \to \kkban(\C, B)$ is natural for bounded algebra homomorphisms: Let $\varphi\colon B \to B'$ be a such a homomorphism. Let $e$ be an idempotent in $B$. Then it is represented in $\kkban(\C, B)$ by the homomorphism $\lambda \mapsto \lambda e$. Composing this homomorphism with $\varphi$ gives the homomorphism $\lambda \mapsto \lambda \varphi(e)$ which happens to represent the idempotent $\varphi(e)$ in $B'$. The obvious extension of these considerations to matrix algebras and formal differences of idempotents gives naturality.

By the universal property of $\kkban$ this transformation is natural with respect to morphisms in $\kkban$ as well.

We construct a map in the converse direction; note that from the Yoneda Lemma this just means that we have to pick an element of $\KTh_0(\C)$: Let $e_{\C}\in \KTh_0(\C)$ be the class of any rank-one idempotent. The functor $\KTh_0$ is half-exact, Morita invariant and homotopy invariant. By the universal property of $\kkban$, the functor $\KTh_0$ factors through $\kkbanfunc$. Thus, we can define a map
$$
\beta_B \colon \kkban(\C,B) \to \KTh_0(B), f \mapsto \beta_B(f):=f_*(e_{\C}).
$$

Note that we have shown that $\alpha \circ \beta$ is a natural transformation from the functor $B \mapsto \kkban(\C,B)$ on $\kkban$ to itself. We can hence use the Yoneda Lemma to identify it with an element of $\kkban(\C,\C)$. To this end, note that $\beta_{\C} ([\id_{\C}]) = e_{\C}$ and $\alpha_{\C}(e_{\C})=[\id_{\C}]$, so $\alpha \circ \beta$ is given by $f \mapsto f^*([\id_{\C}]) = f$, that is, $\alpha_B \circ \beta_B$ is the identity map on $\kkban(\C,B)$.
%

Conversely, if $e \in B$ is an idempotent, let $\varphi\colon \C \to B$ be the induced homomorphism. Then 
$$
\varphi_*([e_{\C}]) = [\varphi(e_{\C})] = [\varphi(1)] = [e] \in \KTh_0(B);
$$
note that we can take $e_{\C}$ to be $1\in \C$.  So $\beta_{B}(\alpha_B([e])) = [e] \in \KTh_0(B)$. Similarly, if $e\in \Mat_n(B)$ is an idempotent, then $\beta_{B}(\alpha_B([e])) = [e] \in \KTh_0(B)$. This settles the claim for unital Banach algebras. By spilt-exactness, it is also true for general Banach algebras. 
\end{proof}

\begin{corollary}
We have $\kk^{\ban}_1(\C,\C) =0$ and $\kk^{\ban}_0(\C, \C) \cong \Z$ with generator $[\id_{\C}]$.
\end{corollary}

\subsection{Action on $\KTh$-theory}

The fact that $\KTh$-theory factors through $\kkban$ means in particular that $\kkban$ acts on $\KTh$-theory and that this action extends the usual functoriality, i.e., the homomorphism $\varphi_*\colon \KTh_0(A) \to \KTh_0(B)$ induced by a continuous homomorphism $\varphi\colon A\to B$ can be interpreted as the action of $\kkbanfunc(\varphi)\in \kkban(A,B)$. The fact that we have a natural isomorphism $\KTh_0(B) \cong \kkbanfunc(\C, B)$ implies that the action of $\varphi\colon A\to B$ on $\KTh$-theory can be represented as multiplication on the right by $\kkbanfunc(\varphi)$. 

Because this observation is important, we give the steps of the construction in some detail for those who do not like to employ as much of the machinery  as above.

Consider the $\KTh$-functor being defined on the category $\BanAlg$ of Banach algebras and homomorphisms. Because $\KTh$-theory is homotopy stable, we obtain an induced functor, which we also call $\KTh$, on the homotopy category $\BanAlg/\!\!\sim$. It can be extended to the category $\SigmaHo$ as follows: If $A$ is a Banach algebra and $m\in \Z$, then define 
$$
\KTh(A,m):= \KTh_{m}(A).
$$ 
If $[\varphi] \in  [\Sigma^{m+k} A, \ \Sigma^{n+k} B]$ represents a morphism from $(A,m)$ to $(B,n)$, with $m+k, n+k\geq 0$, then define
$$
\KTh([\varphi])\colon \KTh_{m}(A) \to \KTh_{n}(B)
$$
by the composition of homomorphisms
$$
\KTh_m(A) \cong \KTh_{-k}(\Sigma^{m+k}A) \stackrel{\KTh_{-k}(\varphi)}{\to}  \KTh_{-k}(\Sigma^{n+k}B) \cong \KTh_{n}(B).
$$
Note that the following diagram commutes
$$
\xymatrix{
&\KTh_{-k}(\Sigma^{m+k}A) \ar[rr]^{\KTh_{-k}(\varphi)} \ar[dd]_{\cong}&& \KTh_{-k}(\Sigma^{n+k}B) \ar[dr]^{\cong} \ar[dd]_{\cong}&\\
 \KTh_m(A) \ar[ru]^{\cong} \ar[rd]^{\cong}&&&& \KTh_{n}(B) \\
 & \KTh_{-k-1}(\Sigma^{m+k+1}A) \ar[rr]^{\KTh_{-k-1}(\Sigma \varphi)} && \KTh_{-k-1}(\Sigma^{n+k+1}B) \ar[ur]^{\cong}
}
$$
Hence, $\KTh$-theory gives a well-defined map on the set $\SigmaHo((A,m), (B,n))$. It is then clear that it actually gives a functor on $\SigmaHo$. 

The morphisms that you invert when moving from $\SigmaHo$ to $\kkban$ are isomorphisms in $\KTh$-theory, so we obtain a $\KTh$-functor on the category $\kkban$ which makes the following diagram commutative:
$$
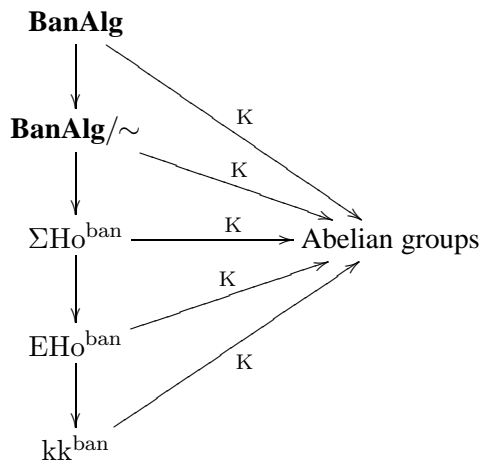
\xymatrix{\BanAlg \ar[rrdd]^{\KTh}\ar[d]&&\\
\BanAlg/\!\!\sim \ar[rrd]^{\KTh}\ar[d]&& \\
\SWban \ar[rr]^{\KTh}\ar[d]&& \text{Abelian groups}\\
\EHoban \ar[rru]^{\KTh}\ar[d]&& \\
\kkban \ar[rruu]_ {\KTh}&& 
}
$$
In particular, isomorphisms in $\BanAlg$ descend to isomorphisms in $\kkban$ which give isomorphisms in $\KTh$-theory.

\subsection{Relation to $\KK$-theory and action of Morita morphisms}

\begin{theorem}
The functor $\kkban \colon \BanAlg \to \kkban$, being Morita invariant and homotopy invariant, factors uniquely through the Morita category $\Moritabancat$. So to every Morita cycle $E \in \Mban(A,B)$ we can assign a $\kkban$-element $\kkban(E) \in \kkban(A,B)$ in a canonical way, and this assignment is functorial. 
\end{theorem}
\begin{proof}
This is a special case of Theorem~\ref{Theorem:LiftToMoritaban}.
\end{proof}

\begin{theorem}
There is a natural transformation $\kkban$ from the bifunctor $(A,B) \to \KKban(A,B)$  to the bifunctor $(A,B) \mapsto \kkban(A,B)$ that takes $[\varphi]$ to $\kkban(\varphi)$ and $[E]$ to $\kkban(E)$ for all homomorphisms of Banach algebras $\varphi$ and for all Morita morphisms $E$. Moreover, the transformation is compatible with the action of the Morita category on $\KKban$ in the second variable.
\end{theorem}
\begin{proof}
Note that the functor $\kkbanfunc$ is Morita invariant, homotopy invariant and half-exact. So, in particular, it is split exact and the constructions of Section~\ref{SectionRelationOfFToKKTheory} apply. 
\end{proof}

\begin{theorem}
There is a unique functor from the category of separable C$^*$-algebras and $\KK$-elements as morphisms to the category $\kkban$ that lifts the functor from the category of separable C$^*$-algebras and $*$-homomorphisms that sends a morphism $\varphi\colon A \to B$ to $\kkbanfunc(\varphi)$. It factors through the natural transformation from $\KKban$ to $\kkban$ described above.
\end{theorem}
\begin{proof}
This follows from Theorem~\ref{Theorem:NaturalTrafoMultiplicative}.
\end{proof}


\begin{thebibliography}{CMR07}

\bibitem[B{\"{u}}h08]{Buehler:08}
Theo B{\"{u}}hler.
\newblock {\em On the Algebraic Foundation of Bounded Cohomology}.
\newblock PhD thesis, ETH Z\"{u}rich, 2008.

\bibitem[CMR07]{CMR:07}
Joachim Cuntz, Ralf Meyer, and Jonathan~M. Rosenberg.
\newblock {\em Topological and bivariant {$K$}-theory}, volume~36 of {\em
  Oberwolfach Seminars}.
\newblock Birkh\"auser Verlag, Basel, 2007.

\bibitem[CT06]{CuntzThom:06}
Joachim Cuntz and Andreas Thom.
\newblock Algebraic {$K$}-theory and locally convex algebras.
\newblock {\em Math. Ann.}, 334:339--371, 2006.

\bibitem[Cun84]{Cuntz:84}
Joachim Cuntz.
\newblock {$K$}-theory and {$C^{\ast} $}-algebras.
\newblock In {\em Algebraic {$K$}-theory, number theory, geometry and analysis
  ({B}ielefeld, 1982)}, volume 1046 of {\em Lecture Notes in Math.}, pages
  55--79. Springer, Berlin, 1984.

\bibitem[Cun87]{Cuntz:87}
Joachim Cuntz.
\newblock A new look at $kk$-theory.
\newblock {\em K-Theory}, 1:31--51, 1987.
\newblock 10.1007/BF00533986.

\bibitem[Cun97]{Cuntz:97}
Joachim Cuntz.
\newblock Bivariante {$K$}-{T}heorie f\"ur lokalkonvexe {A}lgebren und der
  {C}hern-{C}onnes-{C}harakter.
\newblock {\em Doc.\ Math.}, 2:139--182 (electronic), 1997.

\bibitem[Del08]{DellAmbrogio:08}
Ivo Dell'Ambrogio.
\newblock {\em Prime tensor ideals in some triangulated categories of
  {C}$^*$-algebras}.
\newblock PhD thesis, ETH Z\"{u}rich, 2008.

\bibitem[Hig88]{Higson:88}
Nigel Higson.
\newblock Algebraic k-theory of stable c*-algebras.
\newblock {\em Advances in Mathematics}, 67(1):1 -- 140, 1988.

\bibitem[Hig90]{Higson:90:ETheory}
Nigel Higson.
\newblock Categories of fractions and excision in kk-theory.
\newblock {\em Journal of Pure and Applied Algebra}, 65(2):119 -- 138, 1990.

\bibitem[Kas81]{Kasparov:81:Extensions}
Gennadi~G. Kasparov.
\newblock The operator {$K$}-functor and extensions of {$C*$}-algebras.
\newblock {\em Math.\ USSR-Izv.}, 16(3):513--572, 1981.

\bibitem[KV87]{KelVos:87}
Bernhard Keller and Dieter Vossieck.
\newblock Sous les cat\'egories d\'eriv\'ees.
\newblock {\em C. R. Acad. Sci. Paris S\'er. I Math.}, 305(6):225--228, 1987.

\bibitem[Laf02]{Lafforgue:02}
Vincent Lafforgue.
\newblock {$K$}-th\'{e}orie bivariante pour les alg\`{e}bres de {B}anach et
  conjecture de {B}aum-{C}onnes.
\newblock {\em Invent.\ Math.}, 149:1--95, 2002.

\bibitem[Nee01]{Neeman:01}
Amnon Neeman.
\newblock {\em Triangulated categories}, volume 148 of {\em Annals of
  Mathematics Studies}.
\newblock Princeton University Press, Princeton, NJ, 2001.

\bibitem[Par09]{Paravicini:07:Morita:richtigerschienen}
Walther Paravicini.
\newblock Morita equivalences and {KK}-theory for {B}anach algebras.
\newblock {\em J. of the Inst. of Math. of Jussieu}, 8:565--593, 2009.

\bibitem[Par10]{Paravicini:10:GreenJulg:submitted}
Walther Paravicini.
\newblock A generalised {G}reen-{J}ulg theorem for proper groupoids and
  {B}anach algebras.
\newblock {\em submitted}, 2010.

\bibitem[Sch84]{Schochet:84}
Claude Schochet.
\newblock Topological methods for {C}$^*$-algebras. {III}. {A}xiomatic
  homology.
\newblock {\em Pacific J.\ Math}, 114(2):399--445, 1984.

\bibitem[Tho03]{Thom:03}
Andreas~Berthold Thom.
\newblock {\em Connective $E$-theory and bivariant homology for
  {C}$^*$-algebras}.
\newblock PhD thesis, Universit\"{a}t M\"{u}nster, 2003.

\end{thebibliography}
\end{document}